\documentclass[12pt, reqno]{amsart}

\usepackage{amsmath}
\usepackage{amssymb}
\usepackage{amsthm}
\usepackage{mathtools}
\usepackage{mathrsfs}
\usepackage{graphicx}
\usepackage[margin=1.12in]{geometry} 
\usepackage{appendix}
\usepackage{xcolor}
\usepackage[hidelinks,hyperindex,breaklinks]{hyperref}

\usepackage[mathscr]{eucal}

\author{Jianjun Fu}
\address{School of Mathematical Sciences, Peking University,
	No 5. Yiheyuan Road, Beijing, 100871, China}
\email{2601110007@stu.pku.edu.cn}

\author{Shuai Guo}
\address{School of Mathematical Sciences, Peking University,
	No 5. Yiheyuan Road, Beijing, 100871, China}
\email{guoshuai@math.pku.edu.cn}

\author{Qingsheng Zhang}
\address{School of Sciences, Great Bay University,
	Great Bay Institute for Advanced Study,
	Dongguan, 523000, China}
\email{zhangqingsheng@gbu.edu.cn}

\newtheorem{manualtheoreminner}{Theorem}
\newenvironment{manualtheorem}[1]{%
	\renewcommand\themanualtheoreminner{#1}%
	\manualtheoreminner
}{\endmanualtheoreminner}

\newtheorem{dummy}{}[section]
\newtheorem{lemma}[dummy]{Lemma}
\newtheorem{proposition}[dummy]{Proposition}
\newtheorem{theorem}[dummy]{Theorem}

\theoremstyle{definition}

\newtheorem{definition}[dummy]{Definition}

\newtheorem{remark}[dummy]{Remark}

\def\beq{\begin{equation}}
\def\eeq{\end{equation}}

\def\<{\langle}
\def\>{\rangle}

\def\cR{\mathcal{R}}
\def\cT{\mathcal{T}}

\def\cO{\mathcal{O}}

\newcommand{\cq}{y}
\def\Aut{\operatorname{Aut}}

\newcommand\ba{\mathbf a}
\newcommand\bg{\mathbf g}

\newcommand{\bM}{\overline{\mathcal{M}}}

\newcommand{\vir}{\mathrm{vir}}
\newcommand{\C}{\mathbb{C}}
\newcommand{\Q}{\mathbb{Q}}
\newcommand{\Z}{\mathbb{Z}}

\newcommand{\Tot}{{\operatorname{Tot}}}

\newcommand{\bdot}{\mathbin{\boldsymbol{\cdot}}}
\newsavebox{\circledstarbox}
\newcommand{\circledstar}{\mathbin{\mathpalette\circledstaraux\relax}}
\newcommand{\circledstaraux}[2]{%
  \sbox{\circledstarbox}{$#1\star$}%
  \ooalign{\hfil\usebox{\circledstarbox}\hfil\crcr
    \hfil$#1\vcenter{\hbox{%
      \resizebox{1.2\wd\circledstarbox}{!}{$#1\bigcirc$}}}$\hfil\crcr}}
\title[NS blowup equations and local-relative correspondence]
{The inverse local-relative correspondence and \\ the {Nekrasov--Shatashvili} vanishing blowup equation}

\date{}

\begin{document}

\begin{abstract}
	\sloppy
	We prove the Nekrasov--Shatashvili vanishing blowup equation for the local Gromov--Witten theory of smooth weak del Pezzo surfaces, by identifying its unique formal solution with the inverse of the local-relative correspondence. We construct this inverse through an associative set-partition convolution and establish the identification using Zagier's recursion for the Bloch--Okounkov $n$-point function. The proof applies in all genera and includes non-toric surfaces.
\end{abstract}

\maketitle

\setcounter{section}{-1}
\setcounter{tocdepth}{1}
\tableofcontents

\section{Introduction}
	Blowup equations impose functional relations on partition functions and provide constraints on their enumerative invariants. 
	{Early work on blowup equations in enumerative geometry can be traced back to Fintushel and Stern’s study of Donaldson invariants~\cite{FS1996}.}
	In~\cite{NY2005a,NY2005b}, Nakajima and Yoshioka proved blowup equations for cohomological and $K$-theoretic instanton partition functions. 
	Through geometric engineering~\cite{KKV1997}, these equations are related to refined topological strings. Vanishing blowup equations for refined topological strings were proposed in~\cite{GHKK2017}, and in~\cite{HSW2018}, Huang, Sun, and Wang developed a systematic study of both vanishing and unity blowup equations. These equations constrain the refined BPS invariants of local Calabi--Yau threefolds, although their validity remains conjectural in general.

	Let $S$ be a {smooth weak del Pezzo surface}, and let $K_S$ denote its canonical line bundle. The Nekrasov--Shatashvili (NS) limit of the vanishing blowup equation relates the topological and NS specializations of the refined BPS free energy of $\operatorname{Tot}(K_S)$, the associated noncompact Calabi--Yau threefold. These specializations are expected to recover the local Gromov--Witten (GW) theory of $\operatorname{Tot}(K_S)$ and the maximal-contact relative GW theory of $(S,E)$, respectively, where $E\in|-K_S|$ is smooth. The resulting NS vanishing blowup equation therefore predicts a relation between the local and relative GW generating series.
	 These considerations motivate the following \emph{NS vanishing blowup equation}: 
	\begin{equation}\label{eq:blowup-local-relative}
		\sum_{N\in\Z}(-1)^N\exp\Big(-(N+\tfrac12)D \widehat F^{S/E}({\bf Q};\hbar)
		+\widehat F^{K_S}\big(\cq^{(N+\frac12)E}{\bf Q};\mathrm i\hbar\big)\Big)=0.
	\end{equation}
	Here $\widehat F^{K_S}({\bf Q};\hbar)$ and $\widehat F^{S/E}({\bf Q};\hbar)$ denote the sign-twisted all-genus local and relative GW potentials, $\cq=e^{-\hbar}$ with $\hbar$ the string parameter, and $D$ is a differential operator which acts on the Novikov monomial $\mathbf Q^\beta$ of a curve class $\beta$ by
	$D\mathbf Q^\beta=(E\cdot\beta)\mathbf Q^\beta$.
	See~\S \ref{subsec:normalization} for the notation and see Appendix~\ref{app:refined-BPS} for its derivation from the vanishing blowup equation for refined BPS free energy.
	
	On the other hand, in~\cite{BFGW2021}, Bousseau, Fan, Guo, and Wu established a local-relative correspondence in all genera.
	For a smooth projective surface $S$ with a connected smooth nef anticanonical divisor $E$, their correspondence expresses the local GW invariants in terms of maximal-contact relative GW invariants and stationary GW invariants of the elliptic curve $E$. 
	Precisely, for $g\geq 0$,  \vspace{-0.2cm}
	\[
	F_g^{K_S} = (-1)^g F_g^{S/E}+
	\sum_{n\geq 0} \!\!\!\!\sum_{\substack{ g=h+g_1+\dots+g_n,\\ \ba=(a_1,\dots,a_n)\in \Z_{\geq 0}^n, \\ (a_j,g_j) \neq (0,0),\,\sum_{j=1}^n a_j=2h-2} }\!\!\!\!\! 
	\frac{(-1)^{h-1} \widetilde F_{h, \ba}^E}{|\Aut (\ba, \bg)|}
	\prod_{j=1}^n   (-1)^{g_j-1} D^{a_j+2} F_{g_j}^{S/E}\,.
	\] 
	Here \(\widetilde F^E_{h,\mathbf a}\) denotes a shifted stationary GW series of \(E\). 
	In particular, for genus zero and one, the formula reads
	\begin{equation}\label{eqn:LRC-0-and-1-intro}
		\begin{aligned}
			F^{K_S}_0({\bf Q})&=F^{S/E}_0({\bf Q}),\\
			F^{K_S}_1({\bf Q})&=-F^{S/E}_1({\bf Q})
			+\widetilde F^E_{1,\emptyset}(\mathbf q).
		\end{aligned}
	\end{equation}
	Here $\widetilde F^E_{1,\emptyset}$ is the shifted genus-one elliptic GW series {evaluated at ${\bf q}={\bf q}({\bf Q})$ (see~\eqref{eq:elliptic-coordinate-normalized}).} 
	See \S \ref{subsec:local-relative-correspondence} for the notation and conventions.

	This raises the question of whether the local-relative correspondence is equivalent to the NS vanishing blowup equation. The two relations, however, have very different forms.
	We compare these relations using an associative set-partition convolution and a compatible formal action on generating series. 
	We express the two relations in terms of the formal action and show that the corresponding coefficient families are convolution inverses.
	Thus the inverse local-relative formula gives the unique formal solution to the  NS vanishing blowup equation.
	As a consequence, this establishes the  vanishing blowup equation for every smooth weak del Pezzo pair, including non-toric surfaces.

\subsection{Convolution multiplication}\label{subsec:intro-convolution}
We use a convolution on symmetric formal power series families to compare the two relations.
Let $R$ be a commutative $\mathbb Q$-algebra and set
\[
\operatorname{Sym}(R)=\prod_{n\geq1}R[\![z_1,\ldots,z_n]\!]^{S_n}.
\]
For $\mathscr A,\mathscr B\in\operatorname{Sym}(R)$, define
\[
(\mathscr A\star\mathscr B)(z_1,\ldots,z_n)
=\sum_{\{I_1,\ldots,I_k\}\in\mathcal P_n}
\mathscr A(|z_{I_1}|,\ldots,|z_{I_k}|)
\prod_{i=1}^k\mathscr B(z_{I_i}),
\]
where $\mathcal P_n$ is the set of unordered set partitions of $\{1,\ldots,n\}$, $z_I=(z_i)_{i\in I}$, and $|z_I|=\sum_{i\in I}z_i$.
The convolution $\star$ is associative and has the two-sided identity $\mathscr I(z_1,\ldots,z_n)=\delta_{n,1}$.
Moreover, $\mathscr A$ has a unique two-sided inverse under $\star$ if and only if $\mathscr A(z)$ is invertible in $R[[z]]$.
{These properties are proved in Proposition~\ref{prop:convolution-basic-properties}.}

For the application to GW series, we take $R=\mathbb Q[E_2,E_4,E_6]\subset\mathbb Q[[\mathbf q]]$.
A family $\mathscr A\in\operatorname{Sym}(R)$ is \emph{normalized} if, for every $n\geq1$,
\[
\begin{aligned}
\mathscr A(z_1,\ldots,z_n)-\delta_{n,1}&\in z_1\cdots z_nR[\![z_1,\ldots,z_n]\!],\\
\mathscr A(-z_1,\ldots,-z_n)&=(-1)^{n+1}\mathscr A(z_1,\ldots,z_n).
\end{aligned}
\]
It satisfies the \emph{divisor equation} if, for every $n\geq1$,
 	\begin{equation}\label{eq:kernel-divisor-identity}
 		\mathbf q\,\partial_{\mathbf q}\mathscr A(z_1,\ldots,z_n)
 		=[w]\mathscr A(z_1,\ldots,z_n,w).
 	\end{equation}
where $[w]$ denotes the coefficient of $w$.
The normalized families satisfying the divisor equation are closed under $\star$ and convolution inversion.
 
The stationary GW series of the elliptic curve $E$ give the family
\[
\mathscr F(z_1,\ldots,z_n)
:=\delta_{n,1}+(z_1+\cdots+z_n)
\sum_{\substack{g\geq0,\ a_1,\ldots,a_n\geq0\\|\ba|=2g-2}}
F^E_{g,\ba}(\mathbf q)\prod_{i=1}^n z_i^{a_i+1},
\qquad n\geq1,
\]
where $F^E_{g,\ba}(\mathbf q)$ is defined in~\eqref{eq:F_E} and $|\ba|=\sum_i a_i$.
This is a normalized family satisfying the divisor equation~\eqref{eq:kernel-divisor-identity}.

\subsection{Formal action of a normalized family}\label{subsec:intro-formal-action}

We call a series $\mathcal S(\mathbf Q;\lambda)$ \emph{normalized} if
\[
\begin{aligned}
\mathcal S(\mathbf Q;\lambda)&=\sum_{g\geq0}\lambda^{2g-1}\mathcal S_g(\mathbf Q),\qquad
\mathcal S_g(\mathbf Q)-\delta_{g,0}\frac{(\log\mathbf Q^E)^2}{2(E\cdot E)}
&\in\mathbb Q[[\mathbf Q]]\qquad(g\geq0).
\end{aligned}
\]
A normalized family $\mathscr A$ satisfying the divisor equation admit a formal action $\bdot$ on normalized generating series  $\mathcal S(\mathbf Q;\lambda)$. We define \vspace{-0.2cm}
\[
\begin{aligned}
(\mathscr A\bdot\mathcal S)(\mathbf Q;\lambda)
:={}&\frac{\log(\mathbf q)}{\lambda(E\cdot E)}\log(\mathbf Q^E)
+\sum_{n\geq1}\frac{1}{n!}\,
\mathscr A(\widehat z_1,\ldots,\widehat z_n)
\left(\prod_{i=1}^n\widetilde{\mathcal S}(\mathbf Q;\lambda)\right),
\end{aligned}
\]
where \vspace{-0.2cm}
\[
\widetilde{\mathcal S}(\mathbf Q;\lambda)
=\mathcal S(\mathbf Q;\lambda)
-\frac{\log(\mathbf q)}{\lambda(E\cdot E)}\log(\mathbf Q^E).
\]
Here $\mathbf q$ is an auxiliary parameter independent of $\mathbf Q$, and $\widehat z_i=\lambda D$ acts on the $i$-th factor, with $D\mathbf q=0$.
Moreover, this action is compatible with the convolution:
\[
(\mathscr A\star\mathscr B)\bdot\mathcal S
=\mathscr A\bdot(\mathscr B\bdot\mathcal S),
\qquad
\mathscr I\bdot\mathcal S=\mathcal S.
\]

Let $\cT$ and $\cR$ be the normalized derivative series of the local and relative potentials.
The local-relative correspondence can be reformulated as 
	$$
	\cT=\mathscr F\bdot\cR,
	$$
	together with~\eqref{eqn:LRC-0-and-1-intro},
	where $\mathscr F\in \operatorname{Sym}(R)$ encodes the stationary GW series of the elliptic curve.
	This formal action has an inverse $\cR=\mathscr G\bdot\cT$, where $\mathscr G\in \operatorname{Sym}(R)$ is a left inverse of $\mathscr F$: 
	$$\mathscr G\star \mathscr F=\mathscr I.$$
	
On the other hand,
 the NS vanishing blowup equation has a unique solution that can be formulated as
	$$
	\cR=\mathscr V\bdot\cT,
	$$
	where $\mathscr V\in \operatorname{Sym}(R)$ is determined by its recursion relation with the $\Theta$-function.
{The equivalence of the two relations will follow once we show that $\mathscr G=\mathscr V$.}

\subsection{Main results}
  
{Since $\mathscr F(z)=1+O(z^2)$, its convolution inverse $\mathscr G=\mathscr F^{\star-1}$ exists. We first describe the  inverse coefficients and their theta characterization.}

\begin{manualtheorem}{I}\label{main_inverse_theorem}
The unique convolution inverse $\mathscr G=\mathscr F^{\star-1}$ of the   stationary   elliptic curve GW series  has the expansion
\[
\mathscr G(z_1,\ldots,z_n)
=\delta_{n,1}+(z_1+\cdots+z_n)
\sum_{\substack{g\geq1,\;\ba\in\Z_{\geq0}^n\\|\ba|=2g-2}}
G^E_{g,\ba}(\mathbf q)\prod_{i=1}^n z_i^{a_i+1},
\qquad n\geq1.
\]
For $n\geq1$ and $|\ba|=2g-2$, the coefficients $G^E_{g,\ba}$ are quasimodular forms of weight $2g-2+2n$ and satisfy
\[
\mathbf q\frac{d}{d\mathbf q}G^E_{g,\ba}
=G^E_{g,(\ba,0)},
\qquad
\left.G^E_{g,\ba}\right|_{\mathbf q=0}=0\quad(n\geq2).
\]
Furthermore, the generating series $G(z_1,\ldots,z_n)=z_1\cdots z_n\,\mathscr G(z_1,\ldots,z_n)$ is uniquely determined by the following theta relation, for every $n\geq1$:
$$
\sum_{k\geq 0}(-1)^k\sum_{\substack{I_1\sqcup\cdots\sqcup I_k\sqcup J=[n]\\I_1,\ldots,I_k\neq\emptyset}}
\bigg(\prod_{i=1}^k G(z_{I_i})\bigg)\Theta^{(k)}(|z_J|)=0,
$$
where $\Theta(z)=\Theta(z|\tau)$ is the normalized odd theta function defined in~\eqref{eq:delPezzo-Theta-expansion}, with $\mathbf q=e^{2\pi\mathrm i\tau}$ and $\Theta'(0)=1$.
The symbol $\Theta^{(k)}$ denotes its $k$-th derivative with respect to $z$.
The blocks $I_1,\ldots,I_k$ are unordered and counted once, while $J$ may be empty with $|z_\emptyset|=0$.
\end{manualtheorem}
 
The expansion and structural properties are established in Theorem~\ref{thm:inverse-kernel-properties}, 
while the theta characterization follows from Propositions~\ref{prop:solution-of-theta-blowup} and~\ref{prop:F-star-C}.
Moreover, the series \(G^E_{g,\mathbf a}\) serve as the universal coefficients of the inverse local-relative correspondence in Proposition~\ref{prop:general-LR-inverse}.

\begin{remark}
In Appendix~\ref{app:elliptic-stationary}, we give further representations of the generating series \(G\),
including a $\mathbf q$-bracket representation in Proposition~\ref{thm:appq-inverse-bracket}, and two finite rooted-tree expansions in Propositions~\ref{thm:appq-theta-tree} and~\ref{thm:appq-connected-tree}.
\end{remark}
 
We next identify this inverse with the formal solution of the NS vanishing blowup equation.
Using Zagier's recursion for the Bloch--Okounkov $n$-point function~\cite[Theorem~4]{Zagier2016},
	we prove that $\mathscr V$ is the right inverse of $\mathscr F$:
	\[
	\mathscr F\star\mathscr V=\mathscr I.
	\]
By associativity, we have 
	\[
	\mathscr G
	=\mathscr G\star(\mathscr F\star\mathscr V)
	=(\mathscr G\star\mathscr F)\star\mathscr V
	=\mathscr V.
	\]
	Thus, with the genus-zero and genus-one relations~\eqref{eqn:LRC-0-and-1-intro}, the NS vanishing blowup equation is equivalent to the inverse local-relative correspondence.

\begin{manualtheorem}{II}\label{main_theorem}
	Let $S$ be a smooth projective surface and let $E\subset S$ be a connected smooth nef anticanonical divisor with $E^2>0$. With the genus-zero and genus-one relations~\eqref{eqn:LRC-0-and-1-intro}, the local-relative correspondence is equivalent to the NS vanishing blowup equation~\eqref{eq:blowup-local-relative}.
	As a consequence, the NS vanishing blowup equation~\eqref{eq:blowup-local-relative} holds for every smooth weak del Pezzo pair $(S,E)$.
\end{manualtheorem}

\subsection{Outline of the paper}
The paper is organized as follows.
In \S \ref{sec:local-relative correspondence}, we recall the definitions of the local and relative GW invariants and their correspondence. We then normalize the associated potentials and rewrite the NS vanishing blowup equation in theta form.
In \S \ref{sec:inverse_LRC}, we construct the inverse local-relative correspondence using convolution and establish the structural properties of its universal coefficients.
In \S \ref{sec:blowup equation}, we solve the NS vanishing blowup equation and prove its equivalence to the local-relative correspondence.
In \S \ref{sec:examples}, we present three examples, including a non-toric del Pezzo surface.
Appendix~\ref{app:elliptic-stationary} gives further representations of $G$, the generating series associated with the convolution inverse of the stationary elliptic GW family. Appendix~\ref{app:refined-BPS} explains how the refined vanishing blowup equation leads to the NS equation considered here, through its Nekrasov--Shatashvili limit and the  GW identifications.

\subsection*{Acknowledgments}
{We thank Jie Zhou for carefully reading an earlier version of this paper and for many helpful suggestions.}

\subsection*{AI disclosure.}

We used ChatGPT to improve the exposition, assist with extensive
computational tests and checks, and prepare tables.
The mathematical content of this paper was written by the authors,
who take full responsibility for its correctness and for the final manuscript.

$$
$$

\section{Local-relative correspondence and the NS blowup equation}

\label{sec:local-relative correspondence}

We first recall the local and relative GW invariants, their generating series, and the local-relative correspondence for smooth projective surfaces with a connected smooth nef anticanonical divisor.
Starting in \S \ref{subsec:normalization}, we restrict to smooth weak del Pezzo surfaces.
We then rewrite the  NS vanishing blowup equation in theta form.

\subsection{Local-relative correspondence}\label{subsec:local-relative-correspondence}
Let $S$ be a smooth projective surface over $\C$, and let $E\subset S$ be a connected smooth nef anticanonical divisor.
We recall the correspondence between the local GW invariants of the total space $\Tot(\cO_S(-E))\cong\Tot(K_S)$ and the relative GW invariants of the pair $(S,E)$ established by Bousseau--Fan--Guo--Wu~\cite{BFGW2021}.

Let $\beta\in H_2(S,\Z)$ be a curve class such that $\beta\cdot E>0$.
We denote by $\bM_g(\cO_S(-E),\beta)$ the moduli space of genus-$g$ stable maps of class $\beta$ to $\Tot(\cO_S(-E))$,
and by $\bM_g(S/E,\beta)$ the moduli space of genus-$g$ relative stable maps of class $\beta$ to $(S,E)$ with only one contact condition of maximal tangency along $E$.
The moduli space $\bM_g(\cO_S(-E),\beta)$ has virtual dimension $0$, and we define the local invariants by
\[
N_{g,\beta}^{K_S} \coloneqq \deg\!\left([\bM_g(\cO_S(-E),\beta)]^\vir \right).
\]
The moduli space $\bM_g(S/E,\beta)$ has virtual dimension $g$, and we define the relative invariants by
\[
N_{g,\beta}^{S/E} \coloneqq \deg\!\left((-1)^g\lambda_g\cap[\bM_g(S/E,\beta)]^\vir\right),
\]
where $\lambda_g$ is the top Chern class of the Hodge bundle.

For every effective class $\beta\in H_2(S,\Z)$, let $\mathbf Q^\beta$ denote the corresponding monomial in the monoid algebra of effective curve classes, with multiplication $\mathbf Q^\beta\mathbf Q^\gamma=\mathbf Q^{\beta+\gamma}$.  
Following~\cite{BFGW2021}, we introduce generating series:
\begin{equation} \label{eq:F_K_S}
	 F_g^{K_S}({\bf Q}) \coloneqq  \delta_{g,0} {F^{K_S}_{\text{classical}}} ({\bf Q})  + {\delta_{g,1}} F_{\text{unstable}}^{K_S}({\bf Q}) + F_g^{K_S,+}({\bf Q})  \,,
\end{equation}
\begin{equation} \label{eq:F_SE}
	 F_g^{S/E}({\bf Q})  \coloneqq \delta_{g,0} {F^{K_S}_{\text{classical}}} ({\bf Q})  + {\delta_{g,1}} F_{\text{unstable}}^{S/E}({\bf Q}) + F_g^{S/E,+}({\bf Q})  \,,
\end{equation}
where 
\[F_{\text{classical}}^{K_S} ({\bf Q})  \coloneqq -\frac{1-\delta_{(E\cdot E),0}}{3!(E\cdot E)^2}\big(\log \mathbf{Q}^E\big)^3,\]
\[F_{\text{unstable}}^{K_S}({\bf Q})  \coloneqq\left(\frac{1-\delta_{(E\cdot E),0}}{(E\cdot E)}\frac{\chi(S)}{24}-\frac{1}{24}\right)\log \mathbf{Q}^E,\]
\[F_{\text{unstable}}^{S/E}({\bf Q})  \coloneqq - \frac{1-\delta_{(E\cdot E),0}}{(E\cdot E)}\frac{\chi(S)}{24}\log \mathbf{Q}^E ,\]
and the stable part is given by
$$
\begin{aligned}
F_g^{K_S,+}({\bf Q}) &= \sum_{\beta,\, \beta \cdot E>0} N_{g,\beta}^{K_S}\mathbf{Q}^\beta,\\
F_g^{S/E,+}({\bf Q}) &= \sum_{\beta,\,\beta \cdot E>0} \frac{(-1)^{\beta \cdot E+g-1}}{\beta \cdot E}N_{g,\beta}^{S/E}\mathbf{Q}^\beta ,
\end{aligned}
$$
respectively.
Here $\chi(S)$ is the Euler characteristic of $S$ 
and when $E\cdot E=0$, both $\frac{1-\delta_{(E\cdot E),0}}{E\cdot E}$ and
$\frac{1-\delta_{(E\cdot E),0}}{(E\cdot E)^2}$ are, by convention, zero.

The local-relative correspondence relates these two types of generating series using the stationary GW invariants of $E$. By adjunction, $K_E\cong(K_S\otimes\cO_S(E))|_E\cong\cO_E$, so the connected curve $E$ has genus one.
For $g,n\geq0$ and $\ba=(a_1,\ldots,a_n)\in\Z_{\geq0}^n$, we consider the generating series of connected stationary GW invariants of the elliptic curve $E$ given by \vspace{-0.2cm}
\begin{equation}\label{eq:F_E}
F^E_{g,\ba}
:=-\frac{\delta_{g,1}\delta_{n,0}}{24}
\log(\mathbf q)
+\sum_{d\geq0}
\langle\omega\psi^{a_1},\ldots,\omega\psi^{a_n}\rangle^E_{g,n,d}\mathbf q^d.
\end{equation}
Here $\omega\in H^2(E,\Q)$ is the point class. 
See Appendix~\ref{app:elliptic-stationary} for more details.
Moreover, introduce the operator $D$ by
$$
D\mathbf Q^\beta=(\beta\cdot E)\, \mathbf Q^\beta.
$$

The local-relative correspondence theorem states that
\begin{theorem}[{\cite[Theorem 1.2]{BFGW2021}}] \label{thm_locrel_SE}
	For every $g \geq 0$, we have
\[
	F_g^{K_S} = (-1)^g F_g^{S/E}+
	\sum_{n\geq 0} \sum_{\substack{ g=h+g_1+\dots+g_n,\\ \ba=(a_1,\dots,a_n)\in \Z_{\geq 0}^n, \\ (a_j,g_j) \neq (0,0),\,\sum_{j=1}^n a_j=2h-2} } 
	\frac{(-1)^{h-1} \widetilde F_{h, \ba}^E}{|\Aut (\ba, \bg)|}
	\prod_{j=1}^n   (-1)^{g_j-1} D^{a_j+2} F_{g_j}^{S/E}\,.
\] 
Here 
$\widetilde F_{h, \ba}^E= F_{h, \ba}^E-\frac{\delta_{h,1}\delta_{{\bf a},\emptyset}}{24}(E\cdot E)\log(-1)$ and is evaluated at 
\begin{equation}\label{eq:elliptic-coordinate-normalized}
	\mathbf q
	=(-1)^{E\cdot E}{\bf Q}^{E}\exp\bigl(-D^2F_0^{K_S,+}(\mathbf Q)\bigr),
\end{equation}
and $\Aut(\ba,\bg)$ is the subgroup of $S_n$ consisting of permutations $\sigma$ such that $(a_{\sigma(i)},g_{\sigma(i)})=(a_i,g_i)$ for every $i$.
In sums weighted by $|\Aut(\ba,\bg)|^{-1}$, the paired data $(a_i,g_i)_{i=1}^n$ are taken up to simultaneous permutation.
\end{theorem}

\subsection{Normalization of the GW potentials}
\label{subsec:normalization}
From this subsection onward, we assume $E^2>0$.
Define the sign-twisted all-genus potentials by
\begin{equation}\label{eqn:hat-potentials}
\begin{aligned}
\widehat F^{K_S}({\bf Q};\hbar)
&:=\sum_{g\geq0}\hbar^{2g-2}F_g^{K_S}\big((-1)^E{\bf Q}\big),\\
\widehat F^{S/E}({\bf Q};\hbar)
&:=\sum_{g\geq0}\hbar^{2g-1}F_g^{S/E}\big((-1)^E{\bf Q}\big).
\end{aligned}
\end{equation}
	Here the sign substitution acts only on the positive-degree Novikov series, by $\mathbf Q^\beta\mapsto(-1)^{E\cdot\beta}\mathbf Q^\beta$; the classical and genus-one unstable terms remain unchanged.

\begin{definition}
The normalized generating series $\cT_g$ for the local theory and $\cR_g$ for the relative theory are defined by the expansions
\beq\label{def:T_and_R}
\begin{aligned}
\widehat F^{K_S}({\bf Q};\mathrm i\hbar)
&=\sum_{g\geq0}\hbar^{2g-2}\cT_g({\bf Q}),\\
\widehat F^{S/E}({\bf Q};\hbar)
&=-\sum_{g\geq0}\hbar^{2g-1}\cR_g({\bf Q}).
\end{aligned}
\eeq
Furthermore, for $g,n\geq 0$, we define
\begin{equation}\label{def:Dn_T_and_R}
	\cT_{g,n}:=D^n\cT_g,
	\qquad
	\cR_{g,n}:=D^n\cR_g,
\end{equation}
Here the equality $\cT_{0,n}=\cR_{0,n}$, $n\geq 0$, follows from the genus-zero local-relative correspondence.
\end{definition}

By using these notations, {and by straightforward computations,} the local-relative correspondence~{\cite[Theorem 1.2]{BFGW2021}} can be rewritten as follows:
\begin{equation}\label{eqn:LR-correspondence}
	\cT_g=\cR_g+\sum_{n\geq 0}\sum_{\substack{g=h+g_1+\cdots+g_n,\\ \ba=(a_1,\ldots,a_n)\in\Z_{\geq 0}^n,\\ (a_i,g_i)\neq(0,0),\ \sum_{i=1}^n a_i=2h-2}}
	F^E_{h,\ba}\cdot\frac{\prod_{i=1}^n\cR_{g_i,a_i+2}}{|\Aut(\ba,\bg)|},
\end{equation}
For example,
\beq\label{eqn:LRC-0-and-1}
\cT_0=\cR_0,\qquad
\cT_1=\cR_1+F^{E}_{1,\emptyset}.
\eeq

\subsection{Theta-function form of the NS blowup equation}
\label{subsec:geometric-NS-equation}

We rewrite equation~\eqref{eq:blowup-local-relative} in theta form by using the normalized local and relative generating series $\mathcal T_g$ and $\mathcal R_g$ and their derivatives $\mathcal T_{g,n}$ and $\mathcal R_{g,n}$, defined in~\eqref{def:T_and_R} and \eqref{def:Dn_T_and_R}.

Introduce the normalized theta function:
\begin{equation}\label{eq:delPezzo-Theta-expansion}
	\Theta(z)=\Theta(z|\tau)
	=z\, e^{\sum_{k\geq1}\frac{B_{2k}}{2k(2k)!}E_{2k}(\tau)z^{2k}},
\end{equation}
where $B_{2k}$ are Bernoulli numbers and $E_{2k}$ are Eisenstein series normalized to have constant term $1$.
In this subsection, we identify $2\pi {\rm i}\tau=\cT_{0,2}$, so that $e^{2\pi {\rm i}\tau}=\mathbf q$.
We have the following Lemma:
\begin{lemma}
	The NS limit blowup equation~\eqref{eq:blowup-local-relative} is equivalent to the following theta form equation:
	\beq\label{eqn:blowup-Theta}
	\exp\left(\sum_{2g-2+n>0}\hbar^{2g-2+n}\frac{\cT_{g,n}}{n!}\partial_z^n\right)
	\Theta(z)\bigg|_{z=z_{*}}=0,
	\eeq
	where 
	\begin{equation}\label{eq:delPezzo-zNS}
		z_{*}=-\sum_{g\geq1}\hbar^{2g-1}\cR_{g,1}.
	\end{equation}
\end{lemma}
\begin{proof}
By definition, we have
$$
D \widehat F^{S/E}({\bf Q};\hbar)=-\sum_{g\geq 0}\hbar^{2g-1}\cR_{g,1},
$$
and \vspace{-0.2cm}
$$
\widehat F^{K_S}\big(\cq^{-(N+\frac12)E}{\bf Q};\mathrm i\hbar\big)
=\sum_{g,n\geq 0}\hbar^{2g-2+n}\frac{(N+\frac12)^n}{n!}\cT_{g,n}.
$$
Therefore, after the change of summation index $N\to -N-1$, equation~\eqref{eq:blowup-local-relative} can be rewritten as follows:
$$
\sum_{N\in\Z}(-1)^{N}\exp\Big(-(N+\tfrac12)\sum_{g\geq 0}\hbar^{2g-1}\cR_{g,1}
+\sum_{g,n\geq 0}\hbar^{2g-2+n}\frac{(N+\frac12)^n}{n!}\cT_{g,n}\Big)=0.
$$
Notice that $\cR_{0,1}=\cT_{0,1}$. Dividing by $-\sqrt{-1} \, e^{\frac{1}{\hbar^{2}}\cT_{0}+\cT_1}$, the equation can be rewritten as
\beq\label{eqn:blowup-NP}
\exp\bigg(\sum_{2g-2+n>0}\hbar^{2g-2+n}\frac{\cT_{g,n}}{n!}\partial_z^n\bigg)
\theta\Big[\substack{\frac12\\ \frac12}\Big]
\Big(\frac{z}{2\pi\mathrm{i}}\Big|\frac{\cT_{0,2}}{2\pi \mathrm{i}}\Big)\bigg|_{z=z_{*}}=0,
\eeq
where $z_{*}$ is given by \eqref{eq:delPezzo-zNS} and $\theta[\substack{\mu\\ \nu}](w|\tau)$ is the theta-function defined by
\[
\theta[\substack{\mu\\ \nu}](w|\tau)
=\sum_{N\in\Z}\exp\big(\pi\mathrm{i}(N+\mu)^2\tau+2\pi\mathrm{i}(N+\mu)(w+\nu)\big).
\]
The lemma follows from the relation
$\Theta(z)=-\mathrm{i}\,\eta(\tau)^{-3}
	\theta{\Big[\substack{{\frac12}\\ {\frac12}}\Big]}
	\big(\frac{z}{2\pi\mathrm{i}}|\tau \big)$,
where $\eta(\tau)$ is the Dedekind eta-function.
\end{proof}

\begin{remark}
	The left-hand side of~\eqref{eqn:blowup-NP} has the form of a closed non-perturbative partition function as in~\cite[Definition~4.2]{GJZ23}.
	Therefore, the NS limit blowup equation is equivalent to the vanishing of the closed non-perturbative partition function with theta characteristic $(\frac12,\frac12)$.
\end{remark}

$$
$$

\section{Inverting the local-relative correspondence}
\label{sec:inverse_LRC}

We reformulate the local-relative correspondence as a formal action and construct its inverse using an associative set-partition convolution~$\star$ on the elliptic stationary generating series. This formulation allows us to identify the inverse correspondence with the solution of the blowup equation in the next section.

 \subsection{The local-relative correspondence as a formal action}
 	Let $\mathscr A\in \operatorname{Sym}(R)$ be a normalized family introduced in~\S\ref{subsec:intro-convolution} which satisfies the divisor equation \eqref{eq:kernel-divisor-identity}, 
 	and let ${\mathcal S}({\bf Q};\lambda)$ be a normalized series.  
	Recall  the formal action is defined by
 	\beq\label{eqn:A-bdot-S}
 	(\mathscr A \bdot \mathcal S)({\bf Q};\lambda):=\frac{\log({\bf q})}{\lambda(E\cdot E)}\log({\bf Q}^E)+\sum_{n\geq 1}\frac{1}{n!}\,  \mathscr A(\widehat{z_1},\cdots,\widehat{z_n})\,  \bigg(\prod_{i=1}^{n}\widetilde {\mathcal S}({\bf Q};\lambda)\bigg),
 	\eeq
{Here $\widetilde{\mathcal S}$ is the shifted series of~\S\ref{subsec:intro-formal-action}, and $\widehat z_i=\lambda D$ acts on its $i$-th factor.}

 \begin{proposition}\label{prop:well-define-bdot}
 	The formal action $\bdot$ is well defined. Moreover, the resulting series is normalized and independent of ${\bf q}$.
 \end{proposition}
 \begin{proof}
 	By the normalization of $\mathscr A$, we have the following expansion:
 	\[
 	\mathscr A(z_1,\ldots,z_n)
 	=\delta_{n,1}+\sum_{a_1,\ldots,a_n\geq 0}
 	A_{a_1,\ldots,a_n}({\bf q})\prod_{i=1}^n z_i^{a_i+1},
 	\]
 	where $A_{0,\cdots,0}=0$.
 	Equation~\eqref{eqn:A-bdot-S} can be rewritten as follows
 	$$
 	(\mathscr A \bdot \mathcal S)({\bf Q};\lambda)=\mathcal S({\bf Q};\lambda)
 	+\sum_{n\geq 1}\frac{1}{n!}\sum_{a_1,\ldots,a_n\geq 0}
 	A_{a_1,\ldots,a_n}({\bf q}) \bigg(\prod_{i=1}^{n}\widetilde {\mathcal S}^{(a_i+1)}({\bf Q};\lambda)\bigg),
 	$$
 	where $\widetilde {\mathcal S}^{(a+1)}({\bf Q};\lambda)=(\lambda D)^{a+1}\widetilde {\mathcal S}({\bf Q};\lambda)$.
 	By noting that 
 	$$
 	\widetilde {\mathcal S}^{(a+1)}({\bf Q};\lambda)=\delta_{a,0} u+[\mathcal S^{(a+1)}({\bf Q};\lambda)]_{+},
 	$$
 	where $u=D{\mathcal S_0}-\log({\bf q})=\log {\bf Q}^{E}-\log{\bf  q} +O({\bf Q})$ and
 	$$
 	[\mathcal S^{(a+1)}({\bf Q};\lambda)]_{+}=\sum_{2g+a\geq 1}\lambda^{2g+a} D^{a+1} \mathcal S_{g}({\bf Q}),
 	$$
 	we have
 	\begin{align*}  
 		(\mathscr A \bdot \mathcal S)({\bf Q};\lambda)=\mathcal S({\bf Q};\lambda)
 		+ \sum_{\substack{m,k\geq 0\\m+k\geq 1}}\frac{1}{m!} \!\!\!\! \sum_{a_1,\ldots,a_m\geq 0} \!\!\!\!
 		 A_{a_1,\ldots,a_m,0,\cdots,0}({\bf q}) \cdot \frac{u^k}{k!}\cdot \bigg(\prod_{i=1}^{m} [{\mathcal S}^{(a_i+1)}({\bf Q};\lambda)]_{+}\bigg).
 	\end{align*}
 	By the divisor equation, ${\bf q}\frac{d}{d{\bf q}}A_{a_1,\ldots,a_n}({\bf q})=A_{a_1,\ldots,a_n,0}({\bf q})$, and thus
 	$$
 	\sum_{k\geq 0}A_{a_1,\ldots,a_m,0,\cdots,0}({\bf q}) \cdot \frac{u^k}{k!}=A_{a_1,\ldots,a_m}(e^{D\mathcal S_0}).
 	$$
 	This is a power series in ${\mathbf Q}$, since $e^{D\mathcal S_0}$ has positive Novikov order by the normalization of $\mathcal S$.
 	We get
 	\beq\label{eqn:A-bdot-S-Q}
 	(\mathscr A \bdot \mathcal S)({\bf Q};\lambda)=\mathcal S({\bf Q};\lambda)
 	+\sum_{m\geq 1}\frac{1}{m!} \sum_{a_1,\ldots,a_m\geq 0} A_{a_1,\ldots,a_m}(e^{D\mathcal S_0}) 
 	\bigg(\prod_{i=1}^{m} [{\mathcal S}^{(a_i+1)}({\bf Q};\lambda)]_{+}\bigg).
 	\eeq
 	For each fixed power of \(\lambda\), only finitely many terms contribute to the sum in~\eqref{eqn:A-bdot-S-Q}, so the sum is a power series in $\lambda$. 
 	Furthermore, according to the normalization of $\mathscr A$ and $\mathcal S$, it is easy to see $(\mathscr A \bdot \mathcal S)({\bf Q};-\lambda)=-(\mathscr A \bdot \mathcal S)({\bf Q};\lambda)$ .
 \end{proof}

Recall from~\S\ref{subsec:intro-convolution} that, for $n\geq1$,
 \begin{equation}\label{eqn:F_series}
 	\mathscr F(z_1,\ldots,z_n):=\delta_{n,1}+(z_1+\cdots+z_n) \sum_{\substack{g\geq0,\ a_1,\ldots,a_n\geq0\\|\ba|=2g-2}} F^E_{g,\ba}({\bf q})\prod_{i=1}^nz_i^{a_i+1}.
 \end{equation}
 where ${\bf a}=(a_1,\cdots,a_n)$ and $|\ba|=\sum_i a_i$. 
 The function $\mathscr F(z_1,\ldots,z_n)$ is symmetric with respect to $z_i$, since the generating series $F^{E}_{g,\ba}$ is symmetric with respect to $a_i$, $i=1,\cdots,n$,
 and we see $\mathscr F$ is a normalized element in $\operatorname{Sym}(R)$, with $R=\mathbb Q[E_2,E_4,E_6]\subset \mathbb Q[[{\bf q}]]$.
 Furthermore, $\mathscr F$ satisfies the divisor equation according to the divisor equation for the GW theory of elliptic curve.
 
 We introduce 
 $$
 \cT({\bf Q};\lambda):=\sum_{g\geq0}\cT_{g,1}({\bf Q})\lambda^{2g-1},\qquad
 \cR({\bf Q};\lambda):=\sum_{g\geq0}\cR_{g,1}({\bf Q})\lambda^{2g-1}.
 $$
 
 \begin{proposition}\label{prop:LRC-1}
 	Assuming the genus-zero and genus-one relations \eqref{eqn:LRC-0-and-1},
 	the local-relative correspondence~\eqref{eqn:LR-correspondence} is equivalent to the following equation:
 	\begin{equation}\label{eqn:T_in_R}
 		\cT=\mathscr F\bdot \cR.
 	\end{equation}
 \end{proposition}
 
 \begin{proof}
 	In equation~\eqref{eqn:LR-correspondence}, we have used $F^E_{h,\ba}=F^E_{h,\ba}(e^{\mathcal R_{0,2}})$.
 	Now, as in the proof of Proposition~\ref{prop:well-define-bdot}, by using divisor equation, equation~\eqref{eqn:LR-correspondence} can be rewritten as follows:
 	$$
 	\cT_g=\cR_g+\sum_{n\geq 0}\sum_{\substack{g_1+\cdots+g_n+h=g\\ a_1+\cdots+a_n=2h-2}}
 	F^E_{h,\ba}(e^{2\pi i \tau})\cdot\frac{\prod_{i=1}^n\widetilde\cR_{g_i,a_i+2}}{|\Aut(\ba,\bg)|},
 	$$
 	where $\widetilde\cR_{g,a+2}=\cR_{g_i,a_i+2}-\delta_{(g,a),(0,0)}2\pi i \tau$, and $\tau$ is independent of ${\bf Q}$.
 	Taking $D$ action, we get
 	$$
 	\cT_{g,1}=\cR_{g,1}+\sum_{n\geq 0}\sum_{\substack{g_1+\cdots+g_n+h=g\\ a_1+\cdots+a_n=2h-2}}
 	F^E_{h,\ba}(e^{2\pi i \tau})\cdot \frac{D\big(\prod_{i=1}^n\widetilde\cR_{g_i,a_i+2}\big)}{|\Aut(\ba,\bg)|}.
 	$$
 	Then by multiplying $\lambda^{2g-1}$ and taking summation over $g\geq 0$, we obtain
 	$$
 	\cT=\cR+\sum_{n\geq 1}\frac{1}{n!}(\widehat{z_1}+\cdots +\widehat{z_n})\!\!\!\sum_{\substack{g, a_1,\ldots,a_n\geq0\\|\ba|=2g-2}} F^E_{g,\ba}(e^{2\pi i \tau})\prod_{i=1}^n\widehat{z_i}^{a_i+1}\bigg({\prod_{i=1}^n\widetilde \cR}\bigg).
 	$$
 	This is exactly equation~\eqref{eqn:T_in_R}.
 	Here we have used
 	$$
 	\lambda D\bigg({\prod_{i=1}^n\widetilde\cR}\bigg)
 	=(\widehat{z_1}+\cdots +\widehat{z_n})\bigg({\prod_{i=1}^n\widetilde\cR}\bigg).
 	$$ 
 	
 	Conversely, taking the coefficient of $\lambda^{2g-1}$ in~\eqref{eqn:T_in_R} gives the $D$-derivative of~\eqref{eqn:LR-correspondence}. For $g\geq 2$, both sides of~\eqref{eqn:LR-correspondence} have vanishing Novikov constant terms, and $D$ is injective on the terms with $E\cdot\beta>0$. Together with the genus-zero and genus-one relations~\eqref{eqn:LRC-0-and-1}, this recovers the original correspondence.
 \end{proof}
 
\subsection{Convolution multiplication}
Our goal is to find a sequence of generating series $\mathscr G(z_1,\cdots,z_n)$
such that we have the inverse formula of the local-relative correspondence:
\begin{equation}\label{eqn:R_in_T}
	\cR=\mathscr G\bdot \cT.
\end{equation}

For $\mathscr A,\mathscr B\in\operatorname{Sym}(R)$, recall the convolution $\mathscr A\star \mathscr B$ is defined by
	\beq\label{eqn:set-partition-convolution}
	(\mathscr A\star \mathscr B)(z_1,\ldots,z_n)
	:=\sum_{\pi=\{I_1,\ldots,I_k\}\in\mathcal P_n}
	\mathscr A(|z_{I_1}|,\ldots,|z_{I_k}|)\prod_{i=1}^{k}\mathscr B(z_{I_i}),
	\eeq
where $\mathcal P_n$ is the set of unordered set partitions of $\{1,\ldots,n\}$.

\begin{proposition}\label{prop:convolution-basic-properties}
	The convolution $\star$ is associative and has the two-sided identity $\mathscr I$:
	$$
	\mathscr I(z_1,\cdots,z_n)=\delta_{n,1}.
	$$
	Moreover, $\mathscr A$ is invertible under $\star$ if and only if $\mathscr A(z)$ is invertible in $R[[z]]$.
\end{proposition}
\begin{proof}
Both $(\mathscr A\star\mathscr B)\star\mathscr C$ and $\mathscr A\star(\mathscr B\star\mathscr C)$ are indexed by the same two-level partitions of $[n]$. In either expansion, $\mathscr C$ is evaluated on the finest blocks, $\mathscr B$ is evaluated on groups of those blocks, and the resulting block sums are the arguments of $\mathscr A$. The two parenthesizations choose the fine and coarse partitions in the opposite order. Reindexing by this common two-level partition identifies the two sums and proves associativity.

For $\mathscr I\star\mathscr A$, only the partition with one block contributes. For $\mathscr A\star\mathscr I$, only the partition into singletons contributes. Thus $\mathscr I\star\mathscr A=\mathscr A\star\mathscr I=\mathscr A$.

If $\mathscr A$ has an inverse $\mathscr B$, the one-input equation gives $\mathscr B(z)\mathscr A(z)=1$, so $\mathscr A(z)$ is invertible. Conversely, suppose that $\mathscr A(z)$ is invertible. We construct a left inverse from $\mathscr B\star\mathscr A=\mathscr I$. Set $\mathscr B(z)=\mathscr A(z)^{-1}$. For $n\geq2$, isolate the partition $\pi_{\mathrm{disc}}=\{\{1\},\ldots,\{n\}\}$ to obtain
\[
\mathscr B(z_1,\ldots,z_n)
=-\frac{1}{\prod_{i=1}^n\mathscr A(z_i)}
\sum_{\substack{\pi=\{I_1,\ldots,I_k\}\in\mathcal P_n\\
\pi\ne\pi_{\mathrm{disc}}}}
\mathscr B(|z_{I_1}|,\ldots,|z_{I_k}|)
\prod_{i=1}^k\mathscr A(z_{I_i}).
\]
Every $\mathscr B$ on the right has fewer than $n$ inputs, so this determines $\mathscr B$ uniquely by recursion.

Similarly, we construct a right inverse $\mathscr C$ from $\mathscr A\star\mathscr C=\mathscr I$. Its one-input term is $\mathscr C(z)=\mathscr A(z)^{-1}$. For $n\geq2$, the partition with one block contributes
$\mathscr A(z_1+\cdots+z_n)\mathscr C(z_1,\ldots,z_n)$, while every $\mathscr C$ in the remaining terms has fewer than $n$ inputs. Since $\mathscr A(z_1+\cdots+z_n)$ is invertible, this determines $\mathscr C$ recursively. Associativity and the identity give
\[
\mathscr B=\mathscr B\star(\mathscr A\star\mathscr C)
=(\mathscr B\star\mathscr A)\star\mathscr C=\mathscr C.
\]
Hence $\mathscr B$ is a two-sided inverse of $\mathscr A$.
\end{proof}
\begin{remark}
	We can also allow nonempty partial sums of the variables as denominators. Block substitutions preserve this class, so the convolution remains well defined and associative, and the same inverse recursions apply.
\end{remark}
\begin{remark}
In~\cite{Lod10}, Loday gave a polynomial realization of the operad $\operatorname{ComDer}$.
The associativity and identity of $\star$ can also be deduced from this realization by passing to symmetric formal power series families and summing over set partitions.
See also~\cite{BG1975}  for a classical account of set partitions and Möbius inversion.
\end{remark}

\begin{proposition}\label{prop:stable-action-composition}
Let $\mathscr A,\mathscr B$ be normalized families satisfying the divisor equation~\eqref{eq:kernel-divisor-identity}.
Then $\mathscr A\star\mathscr B$ has the same properties, and for every normalized series $\mathcal S$,
\begin{equation}\label{eq:stable-action-composition}
(\mathscr A\star\mathscr B)\bdot \mathcal S
=\mathscr A\bdot(\mathscr B\bdot \mathcal S).
\end{equation}
Moreover, $\mathscr I\bdot\mathcal S=\mathcal S$.
\end{proposition}
\begin{proof}
For $n=1$, we have $(\mathscr A\star\mathscr B)(z)=\mathscr A(z)\mathscr B(z)=1+O(z^2)$.
For $n\geq2$, every variable in a singleton block supplies a factor through $\mathscr A$, while variables in larger blocks supply their factors through $\mathscr B$.
Thus each summand is divisible by $z_1\cdots z_n$.
For a partition with $k$ blocks, simultaneous negation of the variables gives the sign
$(-1)^{k+1}\prod_{i=1}^k(-1)^{|I_i|+1}=(-1)^{n+1}$.
Hence $\mathscr A\star\mathscr B$ is normalized.

For the divisor equation, extract $[w]$ from $(\mathscr A\star\mathscr B)(z_1,\ldots,z_n,w)$.
The singleton block $\{w\}$ contributes ${\bf q}\,\partial_{\bf q}$ on the $\mathscr A$-factor, since $\mathscr B(w)=1+O(w^2)$.
Insertion into a nonempty block contributes ${\bf q}\,\partial_{\bf q}$ on that block's $\mathscr B$-factor; the change in its argument in $\mathscr A$ contributes nothing because $w$ divides $\mathscr B(z_I,w)$.
Their sum is ${\bf q}\,\partial_{\bf q}(\mathscr A\star\mathscr B)$, proving the divisor equation.

Proposition~\ref{prop:well-define-bdot} ensures that the composite action is well defined.
We use the same independent parameter ${\bf q}$ in both actions.
Subtracting the common shift in~\eqref{eqn:A-bdot-S} reduces the composition to ordinary differential substitution on $\widetilde{\mathcal S}$.
Since $D({\bf q})=0$, the Leibniz rule groups the inputs into set partitions, with $\mathscr B$ evaluated on each block and $\mathscr A$ on the corresponding block sums.
This is precisely the convolution~\eqref{eqn:set-partition-convolution}. Adding back the common shift, we get~\eqref{eq:stable-action-composition}.
Finally, the identity $\mathscr I\bdot\mathcal S=\mathcal S$ follows directly from the definition.
\end{proof}

It is clear that $\mathscr F$ is invertible under $\star$ since $\mathscr F(z)$ is invertible in $R[[z]]$.
Thus $\mathscr G=\mathscr F^{\star-1}$ exists. 
We prove that $\mathscr G$ is normalized and satisfies the divisor identity, and thus the inverse formula~\eqref{eqn:R_in_T} follows:
$$
\mathscr G\bdot \cT=\mathscr G\bdot (\mathscr F\bdot \cR)
=(\mathscr G\star \mathscr F)\bdot \cR
=\cR.
$$

\begin{proposition}\label{prop:proper-inverse}
	Let \(\mathscr A\) be a normalized family satisfying the divisor equation. Then \(\mathscr A^{\star-1}\) is normalized and satisfies the same divisor equation.
\end{proposition}
\begin{proof}
Let $\mathscr B=\mathscr A^{\star-1}$. We first prove that $\mathscr B$ is normalized by induction on $n$.
For $n=1$, we have $\mathscr B(z)=\mathscr A(z)^{-1}$, which is even with constant term $1$.
For $n\geq2$, the inverse relation gives
\[
\mathscr B(z_1,\ldots,z_n)
=-\frac{1}{\prod_{i=1}^n\mathscr A(z_i)}
\sum_{\substack{\{I_1,\ldots,I_k\}\in\mathcal P_n\\k<n}}
\mathscr B(|z_{I_1}|,\ldots,|z_{I_k}|)
\prod_{i=1}^k\mathscr A(z_{I_i}).
\]
For $k=1$, the summand is divisible by $z_1\cdots z_n$ through its factor $\mathscr A(z_1,\ldots,z_n)$.
For $k\geq2$, each variable in a non-singleton block supplies a factor through $\mathscr A$, while a singleton block supplies its variable through $\mathscr B$ by induction.
Thus $\mathscr B(z_1,\ldots,z_n)$ is divisible by $z_1\cdots z_n$.
Similarly, induction in the same recursion gives
\[
\mathscr B(-z_1,\ldots,-z_n)
=(-1)^{n+1}\mathscr B(z_1,\ldots,z_n).
\]
Hence $\mathscr B$ is normalized.

We now prove the divisor equation. Consider the coefficient of $w$ in
$(\mathscr B\star\mathscr A)(z_{[n]},w)=0$ for $n\geq1$.
Since $\mathscr A(z_I,w)$ is divisible by $w$ for nonempty $I$, while $\mathscr A(w)=1+O(w^2)$, we obtain
\begin{align*}
0={}&\sum_{\pi=\{I_1,\ldots,I_k\}\in\mathcal P_n}
[w]\mathscr B(|z_{I_1}|,\ldots,|z_{I_k}|,w)
\prod_{i=1}^k\mathscr A(z_{I_i})\\
&+\sum_{\pi=\{I_1,\ldots,I_k\}\in\mathcal P_n}
\mathscr B(|z_{I_1}|,\ldots,|z_{I_k}|)
\sum_{j=1}^k[w]\mathscr A(z_{I_j},w)
\prod_{\substack{1\leq i\leq k\\i\ne j}}\mathscr A(z_{I_i}).
\end{align*}
Now apply ${\bf q}\frac{d}{d{\bf q}}$ to
$(\mathscr B\star\mathscr A)(z_1,\ldots,z_n)=\delta_{n,1}$.
Subtracting the preceding identity and using the divisor equation for $\mathscr A$ gives
\begin{align*}
0={}&\sum_{\pi=\{I_1,\ldots,I_k\}\in\mathcal P_n}
\bigg({\bf q}\frac{d}{d{\bf q}}\mathscr B(|z_{I_1}|,\ldots,|z_{I_k}|)
-[w]\mathscr B(|z_{I_1}|,\ldots,|z_{I_k}|,w)\bigg)
\prod_{i=1}^k\mathscr A(z_{I_i}).
\end{align*}
Since $\mathscr A$ is invertible under $\star$, this implies
\[
{\bf q}\frac{d}{d{\bf q}}\mathscr B(z_1,\ldots,z_n)
=[w]\mathscr B(z_1,\ldots,z_n,w),
\]
which is the divisor equation for $\mathscr B$.
\end{proof}

\subsection{Inverse formula of the local-relative correspondence}
By the properties of the convolution product \(\star\) and its compatibility with the formal action \(\boldsymbol{\cdot}\), we obtain the inverse local-relative correspondence in the form \(\mathcal R=\mathscr G\boldsymbol{\cdot}\mathcal T\), where \(\mathscr G=\mathscr F^{\star-1}\). Equivalently, \(\mathscr G\) is uniquely determined by
\beq\label{eqn:G-F}
\sum_{\pi=\{I_1,\ldots,I_k\}\in\mathcal P_n}\mathscr G(|z_{I_1}|,\ldots,|z_{I_k}|)\prod_{i=1}^{k}\mathscr F(z_{I_i})
=\delta_{n,1}.
\eeq
We next show that \(\mathscr G\) admits an expansion analogous to that of \(\mathscr F\) and satisfies the following properties.

\begin{theorem}\label{thm:inverse-kernel-properties}
For every $n\geq1$, the series $\mathscr G(z_1,\ldots,z_n)$ admits the following expansion
\begin{equation}\label{eqn:G_series}
\mathscr G(z_1,\ldots,z_n)=\delta_{n,1}+(z_1+\cdots+z_n)
\sum_{\substack{g\geq0,\ a_1,\ldots,a_n\geq0\\|\ba|=2g-2}}
G^E_{g,\ba}\prod_{i=1}^nz_i^{a_i+1}.
\end{equation}
Furthermore, we define $G^E_{g,\ba}$ by~\eqref{eqn:G_series} with extension $G^E_{1,\emptyset}:=-F^E_{1,\emptyset}$.
We set $G^E_{g,\ba}=0$ for $|\ba|\ne2g-2$.
For $n=\ell(\ba)\geq1$, the coefficients $G^E_{g,\ba}$ are quasimodular forms of weight $|\ba|+2n$.
The coefficients satisfy the divisor equation
\begin{equation}\label{eq:divisor-G}
{\mathbf q}\frac{d}{d{\mathbf q}}G^E_{g,\ba}=G^E_{g,(\ba,0)},
\end{equation}
and
\begin{equation}\label{eq:G-0}
G^E_{g,\ba}|_{\mathbf q=0}=0,\qquad \ell(\ba)\geq2.
\end{equation}
\end{theorem}
\begin{proof}
By Proposition~\ref{prop:proper-inverse}, $\mathscr G=\mathscr F^{\star-1}$ is normalized and satisfies the divisor equation.
It remains to prove the additional factor $z_1+\cdots+z_n$ in~\eqref{eqn:G_series}.
For $n=1$, normalization gives $\mathscr G(z)-1\in z^2R[[z]]$.
For $n\geq2$, we use induction and the inverse recursion
\[
\mathscr G(z_1,\ldots,z_n)
=-\frac{1}{\prod_{i=1}^n\mathscr F(z_i)}
\sum_{\substack{\{I_1,\ldots,I_k\}\in\mathcal P_n\\k<n}}
\mathscr G(|z_{I_1}|,\ldots,|z_{I_k}|)
\prod_{i=1}^k\mathscr F(z_{I_i}).
\]
For $k=1$, the summand contains the factor $z_1+\cdots+z_n$ through $\mathscr F(z_1,\ldots,z_n)$.
For $k\geq2$, it contains the same factor through $\mathscr G(|z_{I_1}|,\ldots,|z_{I_k}|)$ by induction, since $|z_{I_1}|+\cdots+|z_{I_k}|=z_1+\cdots+z_n$.
Thus $\mathscr G(z_1,\ldots,z_n)$ is divisible by $z_1+\cdots+z_n$.
Together with the divisibility by $z_1\cdots z_n$ and the parity supplied by normalization, this shows that
\[
\frac{\mathscr G(z_1,\ldots,z_n)-\delta_{n,1}}
{(z_1+\cdots+z_n)z_1\cdots z_n}
\]
is a symmetric formal power series containing only monomials of even total degree.
Writing this degree as $2g-2$, we get~\eqref{eqn:G_series}.

For nonempty $\ba$, comparing coefficients in the divisor equation for $\mathscr G$ gives~\eqref{eq:divisor-G}.
For the empty tuple,
\[
\mathbf q\frac{d}{d\mathbf q}G^E_{1,\emptyset}
=-F^E_{1,(0)}=G^E_{1,(0)},
\]
where the last equality follows from $\mathscr G(z)=\mathscr F(z)^{-1}$.

Recall that $\mathscr F(z_1,\ldots,z_n)|_{\mathbf q=0}=0$ for $n\geq2$, since degree-zero stationary invariants with at least two point insertions vanish.
Evaluating~\eqref{eqn:G-F} at $\mathbf q=0$ therefore gives
\[
0=\left.\mathscr G(z_1,\ldots,z_n)\right|_{\mathbf q=0}
\prod_{i=1}^n\left.\mathscr F(z_i)\right|_{\mathbf q=0},
\qquad n\geq2.
\]
Since $\mathscr F(z)|_{\mathbf q=0}=1+O(z^2)$ is invertible, we obtain $\mathscr G(z_1,\ldots,z_n)|_{\mathbf q=0}=0$ for $n\geq2$.
Comparing coefficients in~\eqref{eqn:G_series} proves~\eqref{eq:G-0}.

Finally, in~\cite[Section~5]{OP2006}, Okounkov--Pandharipande showed that $F^E_{g,\ba}$ is a quasimodular form of weight $|\ba|+2n$, where $n=\ell(\ba)\geq1$.
Thus a coefficient of total $z$-degree $d$ in $\mathscr F(z_1,\ldots,z_n)$ has weight $d+n-1$.
This grading is preserved by convolution and the inverse recursion above, so the same holds for $\mathscr G$.
Comparing with~\eqref{eqn:G_series}, we see the weight is $|\ba|+2n$ for $G^E_{g,\ba}$.
\end{proof}

Now we write down the explicit inverse formula of the local-relative correspondence.
\begin{proposition}\label{prop:general-LR-inverse}
	For every $g\geq0$, the inverse local-relative correspondence is
	\begin{equation}\label{eqn:local-relative-inverse}
		\cR_g=\cT_g
		+\sum_{n\geq0}\sum_{\substack{g=h+g_1+\cdots+g_n\\\ba=(a_1,\ldots,a_n)\in\Z_{\geq0}^n,\\ (a_i,g_i)\ne(0,0),\, \sum_{i=1}^na_i=2h-2}}
		G^E_{h,\ba}\frac{\prod_{i=1}^n\cT_{g_i,a_i+2}}{|\Aut(\ba,\bg)|},
	\end{equation}
	where \(G^E_{h,\mathbf a}\) is evaluated at \(\mathbf q=e^{\mathcal T_{0,2}}\).
\end{proposition}
\begin{proof}
For genus $0$ and $1$, the formula is $\cR_0=\cT_0$ and $\cR_1=\cT_1-F^E_{1,\emptyset}$, as in~\eqref{eqn:LRC-0-and-1}.
For $g\geq2$, we apply the same argument as in the proof of Proposition~\ref{prop:LRC-1}, with $F^E_{h,\ba}$ replaced by $G^E_{h,\ba}$ and $\cR,\cT$ interchanged, using the divisor equation and vanishing property in Theorem~\ref{thm:inverse-kernel-properties}.
Thus~\eqref{eqn:local-relative-inverse} is equivalent to $\cR=\mathscr G\bdot\cT$, the convolution inversion of $\cT=\mathscr F\bdot\cR$.
\end{proof}

$$
$$

\section{Equivalence between blowup equation and local-relative correspondence}
\label{sec:blowup equation}

In this section we prove Theorem~\ref{main_theorem}. We first solve the theta function equation \eqref{eqn:blowup-Theta} and then identify its solution with the inverse local-relative formula.

\subsection{Solving the NS limit blowup equation}
\label{subsec:solve-NS}
In this subsection, we express the solution $z_{*}=-\sum_{g\geq1}\hbar^{2g-1}\cR_{g,1}$ of equation~\eqref{eqn:blowup-Theta} in term of $\cT_{g,n}$.

\begin{proposition}\label{prop:solution-of-theta-blowup}
	Equation~\eqref{eqn:blowup-Theta} has a unique formal solution $z_{*}=-\sum_{g\geq1}\hbar^{2g-1}\cR_{g,1}$, which has the form
	\beq\label{eqn:sol-blowup}
	\cR_{g,1}=\sum_{n\geq 1}\sum_{\substack{h+ g_1+\cdots+g_n=g\\ \mu_j\geq0,\, 2g_j-1+\mu_j>0 \\ |\boldsymbol\mu|=2h-1+n}}
	V_{h,\boldsymbol\mu}\cdot\frac{\prod_{j=1}^n\cT_{g_j,\mu_j+1}}{|\Aut(\mathbf g,\boldsymbol\mu)|},
	\qquad g\geq 1.
	\eeq
	Moreover, the coefficients $V_{h,\boldsymbol\mu}$ are determined in the following way. Let
	\begin{equation}\label{eq:theta-solution-kernel}
		V(z_1,\ldots,z_n)
		=\sum_{\substack{h\geq0,\;\mu_i\geq0\\|\boldsymbol\mu|=2h-1+n}}
		V_{h,\boldsymbol\mu}z_1^{\mu_1+1}\cdots z_n^{\mu_n+1},
	\end{equation}
	then $V\in \operatorname{Sym}(R)$ satisfies and is uniquely determined by the following relation:
	\beq\label{eqn:C-Theta}
	\sum_{k\geq 0}(-1)^k\sum_{\substack{I_1\sqcup\cdots\sqcup I_k\sqcup J=[n]\\I_1,\ldots,I_k\neq\emptyset}}
	\bigg(\prod_{i=1}^kV(z_{I_i})\bigg)\Theta^{(k)}(|z_J|)=0.
	\eeq
	Here $\Theta^{(k)}$ denotes the $k$-th derivative of $\Theta$, and the sets $I_1,\ldots,I_k$ are unordered and counted once.
\end{proposition}

\begin{proof}
	We first divide out the invertible scalar
	$\exp\big(\sum_{g\geq2}\hbar^{2g-2}\cT_{g,0}\big)$ coming from the terms with no $z$-derivative.
	Then by taking expansion of the equation, we have
	\beq\label{eqn:R-T-Theta}
	0=\sum_{n,k\geq0}\frac{1}{n!k!} \!\!
	\sum_{\substack{g_j,\mu_j\geq0\\2g_j-1+\mu_j>0}} \!\!
	\bigg(\prod_{j=1}^n\frac{\cT_{g_j,\mu_j+1}}{(\mu_j+1)!}\hbar^{2g_j-1+\mu_j}\bigg)
	\bigg(-\sum_{h\geq1}\cR_{h,1}\hbar^{2h-1}\bigg)^k
	\Theta^{(k+|\boldsymbol\mu|+n)}(0).
	\eeq
	Since $\Theta$ is odd, every nonzero term has odd degree in $\hbar$.
	In the coefficient of $\hbar^{2g-1}$, the term $\cR_{g,1}$ occurs only as $-\cR_{g,1}\Theta'(0)=-\cR_{g,1}$.
	All other terms involve only $\cT$ and $\cR_{h,1}$ with $h<g$.
	The equation therefore determines the solution uniquely.
	
	To determine the coefficients, we note first that~\eqref{eqn:C-Theta} can be rewrite in a recursion form:
	\beq\label{eqn:recursion-C}
	V(z_1,\ldots,z_n)
	=\Theta(z_1+\cdots+z_n)+\sum_{k\geq1}(-1)^k
	\sum_{\substack{I_1\sqcup\cdots\sqcup I_k\sqcup J=[n]\\I_1,\ldots,I_k\neq\emptyset\\J\neq\emptyset\text{ if }k=1}}
	\bigg(\prod_{i=1}^kV(z_{I_i})\bigg)\Theta^{(k)}(|z_J|).
	\eeq
	Every block on the right-hand side is a proper subset of $[n]$, so the recursion uniquely determines $V$, starting with $V(z)=\Theta(z)$.
	By using equations~\eqref{eq:theta-solution-kernel} and \eqref{eqn:recursion-C}, and by substituting \eqref{eqn:sol-blowup} into \eqref{eqn:R-T-Theta}, the equation holds true.
	The uniqueness gives us that ~\eqref{eqn:sol-blowup} is the unique solution of equation~\eqref{eqn:blowup-Theta}. 

It remains to prove that the functions $V(z_1,\cdots,z_n)$ determined by~\eqref{eqn:C-Theta} do have the expansion~\eqref{eq:theta-solution-kernel}.
This can be proved as follows.
Firstly, we have by induction
$$
V(-z_1,\cdots,-z_n)=-V(z_1,\cdots,z_n).
$$
Secondly, fixing an index $i\in[n]$ and setting $z_i=0$ in~\eqref{eqn:recursion-C}, the sum of the terms with $i\in J$ vanishes by the $(n-1)$-variable case of equation~\eqref{eqn:C-Theta}.
Applying this to equation~\eqref{eqn:C-Theta} and by induction, we see $V(z_1,\ldots,z_n)$ is divisible by $z_1\cdots z_n$. 
Lastly, for $n\geq 2$, applying $\partial_{z_j}$ on equation~\eqref{eqn:recursion-C}, we get
$$
\partial_{z_j}V(z_1,\ldots,z_n)
=\Theta'(z_1+\cdots+z_n)+\sum_{k\geq1}(-1)^k
\sum_{\substack{I_1\sqcup\cdots\sqcup I_k\sqcup J=[n]\\I_1,\ldots,I_k\neq\emptyset\\J\neq\emptyset\text{ if }k=1}}
\partial_{z_j}\bigg[\bigg(\prod_{i=1}^kV(z_{I_i})\bigg)\Theta^{(k)}(|z_J|)\bigg].
$$
Setting $z_j=0$, we prove the vanishing of this derivative by induction on $n$.
If $j\in I_i$ for some $|I_i|\geq 2$, then the term vanishes by induction.
The remaining terms come from $j\in J$ or $I_i=\{j\}$ for which we have $V'(0)=1$.
Therefore we have
$$ \small{
\begin{aligned}
\left.\partial_{z_j}V(z_1,\ldots,z_n)\right|_{z_j=0}
={}&\sum_{k\geq0} 
\sum_{\substack{I_1\sqcup\cdots\sqcup I_k\sqcup J=[n]\setminus\{j\}\\I_1,\ldots,I_k\neq\emptyset}}
\!\!\!\! \big((-1)^{k+1}+(-1)^k\big) \bigg(\prod_{a=1}^kV(z_{I_a})\bigg)\Theta^{(k+1)}(|z_J|)=0.
\end{aligned}
}
$$
Here the two signs correspond to $j$ in some $I_i$ and in $J$, respectively.
Thus $\prod_i z_i^2$ divides $V(z_1,\ldots,z_n)$ for $n\geq2$.
Putting all these together, we get the desired expansion.
\end{proof}

\begin{remark}
We have $V_{h,(0)}=\delta_{h,0}$ and $V_{h,(i,j)}=0$ for $i=0$ or $j=0$.
\end{remark}
We show some examples. After dividing by the same invertible scalar as in the proof, we expand~\eqref{eqn:blowup-Theta} using~\eqref{eq:delPezzo-Theta-expansion}. 
In particular,
\label{par:low-genus-theta}
\[
\Theta'(0\mid\tau)=1,\qquad
\Theta'''(0\mid\tau)=\frac{E_2(\tau)}4.
\]
The coefficient of $\hbar$ is
$-\cR_{1,1}+\cT_{1,1}+\frac{E_2(\tau)}{24}\cT_{0,3}$, and hence
\begin{equation}\label{eq:examples-genus-one}
\cR_{1,1}=\cT_{1,1}+\frac{E_2(\tau)}{24}\cT_{0,3}.
\end{equation}
To compute the coefficient of $\hbar^3$, use
\eqref{eq:examples-genus-one} to combine the translation by $z_*$
with the differential operator in~\eqref{eqn:blowup-Theta}.
The exponent, through order $\hbar^3$, is
\[ \small
\hbar\frac{\cT_{0,3}}6
\left(\partial_z^3-\frac{E_2(\tau)}4\partial_z\right)
+\hbar^2\left(\frac{\cT_{1,2}}2\partial_z^2
+\frac{\cT_{0,4}}{24}\partial_z^4\right) +\hbar^3\left(
(\cT_{2,1}-\cR_{2,1})\partial_z
+\frac{\cT_{1,3}}6\partial_z^3
+\frac{\cT_{0,5}}{120}\partial_z^5\right).
\]
Its exponential has three contributions of order $\hbar^3$:
the cubic-order term, the product of the first- and second-order
terms, and one sixth of the cube of the first-order term.
Applying these operators to $\Theta$ at $z=0$ and by using~\eqref{eq:delPezzo-Theta-expansion},
where $E_8=E_4^2$, we get
\begin{equation}\label{eq:examples-genus-two}
\begin{aligned}
\cR_{2,1}={}&\cT_{2,1}
+\frac{E_2}{24}\cT_{1,3}
+\frac{5E_2^2-2E_4}{5760}\cT_{0,5} +\frac{E_2^2-E_4}{288}\cT_{0,3}\cT_{1,2}\\
&+\frac{5E_2^3-9E_2E_4+4E_6}{20736}
\cT_{0,3}\cT_{0,4} +\frac{E_2^4-4E_2^2E_4+4E_2E_6-E_4^2}{82944}\cT_{0,3}^3,
\end{aligned}
\end{equation}
Here all Eisenstein series are evaluated at $2\pi\mathrm i\tau=\cT_{0,2}$.

\subsection{Solution of the NS limit blowup equation as a formal action}
In this subsection, we express equation~\eqref{eqn:sol-blowup} as a formal action.
To obtain this, we first establish the analogue of the divisor equation for $V_{h,\boldsymbol\mu}$.

\begin{lemma}\label{lem:theta-kernel-divisor}
We have 
$$
\mathbf q\frac{d}{d\mathbf q}V_{h,\boldsymbol\mu}
=V_{h,(\boldsymbol\mu,1)}.
$$
\end{lemma}
\begin{proof}
This is equivalent to 
	\begin{equation}\label{eq:theta-kernel-divisor}
		\mathbf q\frac{d}{d\mathbf q}V(z_1,\cdots,z_n)
		=[w^2]V(z_1,\cdots,z_n,w).
	\end{equation}
Apply $\mathbf q\frac{d}{d\mathbf q}$ to~\eqref{eqn:C-Theta} and by using
$\mathbf q\frac{d}{d\mathbf q}\Theta(z)=\frac12\Theta''(z)-\frac{E_2}{8}\Theta(z)$,
we get
\begin{align*}
0={}&\sum_{k\geq1}(-1)^k
\sum_{\substack{I_1\sqcup\cdots\sqcup I_k\sqcup J=[n]\\I_1,\ldots,I_k\neq\emptyset}}
\sum_{a=1}^k
\bigg(\mathbf q\frac{d}{d\mathbf q}V(z_{I_a})\bigg)
\prod_{\substack{1\leq i\leq k\\i\neq a}}V(z_{I_i})
\Theta^{(k)}(|z_J|)\\
&+\frac12\sum_{k\geq0}(-1)^k
\sum_{\substack{I_1\sqcup\cdots\sqcup I_k\sqcup J=[n]\\I_1,\ldots,I_k\neq\emptyset}}
\bigg(\prod_{i=1}^kV(z_{I_i})\bigg)\Theta^{(k+2)}(|z_J|).
\end{align*}
Next, take the coefficient of $w^2$ in~\eqref{eqn:C-Theta} for the variables $z_1,\cdots,z_n,w$.
If $w$ belongs to $J$, its contribution is the second sum above.
If $\{w\}$ is a block, its contribution vanishes since $V(w)=\Theta(w)=w+O(w^3)$.
Otherwise, $w$ joins a nonempty block $I_a$, and the corresponding factor is $[w^2]V(z_{I_a},w)$.
Subtracting the two equations, we obtain
$$
0=\sum_{k\geq1}(-1)^k
\sum_{\substack{I_1\sqcup\cdots\sqcup I_k\sqcup J=[n]\\I_1,\ldots,I_k\neq\emptyset}}
\sum_{a=1}^k\Delta(z_{I_a})
\prod_{\substack{1\leq i\leq k\\i\neq a}}V(z_{I_i})
\Theta^{(k)}(|z_J|),
$$
where $\Delta(z_1,\cdots,z_n):=\mathbf q\frac{d}{d\mathbf q}V(z_1,\cdots,z_n)-[w^2]V(z_1,\cdots,z_n,w)$.
For $n=1$, this gives $-\Theta'(0)\Delta(z_1)=-\Delta(z_1)=0$.
Inductively, every term involving a proper block vanishes, leaving only
$-\Delta(z_1,\cdots,z_n)=0$.
This proves~\eqref{eq:theta-kernel-divisor} and thus the Lemma.
\end{proof}

We introduce $\mathscr V\in\operatorname{Sym}(R)$ by
\[
\mathscr V(z_1,\ldots,z_n;\mathbf q)
:=\frac{V(z_1,\ldots,z_n;\mathbf q)}{z_1\cdots z_n}.
\]
Proposition~\ref{prop:solution-of-theta-blowup} and Lemma~\ref{lem:theta-kernel-divisor} show that $\mathscr V$ is normalized and satisfies~\eqref{eq:kernel-divisor-identity}.

\begin{proposition}\label{prop:blowup-formal-action}
Together with $\cT_{0,1}=\cR_{0,1}$, the NS limit vanishing blowup equation (equivalently, its theta form~\eqref{eqn:blowup-Theta} or solution form~\eqref{eqn:sol-blowup}) is equivalent to
\begin{equation}\label{eq:blowup-formal-action}
\cR=\mathscr V\bdot\cT.
\end{equation}
\end{proposition}
\begin{proof}
This follows form a similar discussion as the proof of Proposition~\ref{prop:well-define-bdot}
by applying~\eqref{eqn:A-bdot-S-Q} with $\mathscr A=\mathscr V$ and $\mathcal S=\cT$.
By substituting the expansion~\eqref{eq:theta-solution-kernel} and comparing coefficients of $\lambda^{2g-1}$, we get~\eqref{eqn:sol-blowup} for $g\geq1$.
The genus-zero term is the assumed equality $\cT_{0,1}=\cR_{0,1}$.
\end{proof}

\subsection{Relation between blowup equation and local-relative correspondence}
\label{subsec:NS-local-relative-comparison}
To complete the proof of Theorem~\ref{main_theorem}, it remains to identify $\mathscr V$ with the convolution inverse $\mathscr G$.

We first do some preparations. For $\mathscr A\in \operatorname{Sym}(R)$, define its disconnected version by
$$
\mathscr A^\bullet(z_1,\ldots,z_n)
:=\sum_{\{I_1,\ldots,I_k\}\in\mathcal P_n}
\prod_{i=1}^k\mathscr A(z_{I_i}),
\qquad \mathscr A^\bullet(z_\emptyset):=1.
$$
We also note that the disconnected operation is invertible. We have
$$
\mathscr A(z_1,\ldots,z_n)
=\sum_{\{I_1,\ldots,I_k\}\in\mathcal P_n}
(-1)^{k-1}(k-1)!\prod_{i=1}^k\mathscr A^\bullet(z_{I_i}),
\qquad n\geq1.
$$

\begin{proposition}\label{prop:circ-to-bullet}
We have $\mathscr I^\bullet(z_1,\ldots,z_n)=1$, $n\geq1$.
Furthermore, for $\mathscr A,\mathscr B\in\operatorname{Sym}(R)$,
$$
(\mathscr A\star \mathscr B)^\bullet=\mathscr A^\bullet\star \mathscr B. 
$$
\end{proposition}
\begin{proof}
For the first part, by definition, we have immediately
	\beq
	\mathscr I^\bullet(z_1, \dots, z_n) = \sum_{\pi \in \mathcal{P}_n} \prod_{J \in \pi} \delta_{|J|,1}=1.
	\eeq
For the second part, let $\mathcal P(J)$ denote the set of set partitions of $J$. By definition,
\begin{align*}
&\, (\mathscr A\star \mathscr B)^\bullet(z_1,\ldots,z_n)\\
=&\, \sum_{\{J_1,\ldots,J_k\}\in\mathcal P_n}
\sum_{\substack{\{I_{j,1},\cdots,I_{j,k_j}\}\in\mathcal P(J_j)\\1\leq j\leq k}}
\bigg(\prod_{j=1}^{k}
\mathscr A(|z_{I_{j,1}}|,\ldots,|z_{I_{j,k_j}}|)\bigg)
\bigg(\prod_{j=1}^{k}\prod_{i=1}^{k_j}
\mathscr B(z_{I_{j,i}})\bigg).
\end{align*}
Fix a partition $\{I_1,\ldots,I_l\}$ of $\{1,\cdots,n\}$ formed by all of the blocks $I_{j,i}$. 
Summing over all ways to group these blocks into $J_1,\ldots,J_k$ gives $\mathscr A^\bullet(|z_{I_1}|,\ldots,|z_{I_k}|)$. Therefore,
\begin{align*}
(\mathscr A\star \mathscr B)^\bullet(z_1,\ldots,z_n)
=\sum_{\{I_1,\ldots,I_k\}\in\mathcal P_n}
\mathscr A^\bullet(|z_{I_1}|,\ldots,|z_{I_k}|)
\prod_{i=1}^{k}\mathscr B(z_{I_i})
\end{align*}
which equals to exactly $(\mathscr A^\bullet\star \mathscr B)(z_1,\ldots,z_n)$.
\end{proof}

Now we are ready to prove that $\mathscr V$ is the right $\star$-inverse of $\mathscr F$.
\begin{proposition}\label{prop:F-star-C}
The following equation holds:
\beq\label{eqn:relation-F-and-C}
\mathscr F\star \mathscr V=\mathscr I.
\eeq
As a corollary, by the associativity of $\star$, we have
$$
\mathscr G=\mathscr G\star(\mathscr F\star \mathscr V)=(\mathscr G\star \mathscr F)\star \mathscr V=\mathscr V.
$$
\end{proposition}

\begin{proof}
We introduce the series $F(z_1,\cdots,z_n)$ by
$$
\mathscr F(z_1,\ldots,z_n)=(z_1+\cdots+z_n)F(z_1,\ldots,z_n).
$$
Then we have
$$
(\mathscr F\star \mathscr V)(z_1,\ldots,z_n)
=\frac{z_1+\cdots+z_n}{z_1\cdots z_n}(F\star V)(z_1,\ldots,z_n).
$$
Since this factor is nonzero and equals $1$ for $n=1$, equation~\eqref{eqn:relation-F-and-C} is equivalent to $F\star V=\mathscr I$.
By Proposition~\ref{prop:circ-to-bullet} and the invertibility of the disconnected operation, this is equivalent to
\begin{equation}\label{eq:C-star-Fbullet}
(F^\bullet\star V)(z_1,\ldots,z_n)=1.
\end{equation}

We prove equation~\eqref{eq:C-star-Fbullet} by induction on $n$.
Recall the following recursive form of the Bloch--Okounkov formula proved by Zagier~\cite{Zagier2016}:
\begin{equation}\label{Zagier_recursive_relation}
\sum_{I\subseteq[m]}(-1)^{m-|I|}
\Theta^{(m-|I|)}(|z_I|)F^\bullet(z_I)=0,
\qquad m\geq1.
\end{equation}
For $n=1$, equation~\eqref{Zagier_recursive_relation} gives $F^\bullet(z_1)\Theta(z_1)=\Theta'(0)=1$.
Since $V(z_1)=\Theta(z_1)$, we have $(F^\bullet\star V)(z_1)=1$.
For each partition $\{I_1,\ldots,I_m\}\in\mathcal P_n$, we apply equation~\eqref{Zagier_recursive_relation} to the variables  $|z_{I_1}|,\allowbreak\ldots,\allowbreak|z_{I_m}|$, multiply by $\prod_{i=1}^m V(z_{I_i})$, and sum over $\mathcal P_n$.
In each term, denote the blocks occurring in the arguments of $F^\bullet$ by $J_1,\ldots,J_\ell$, and the remaining blocks by $I_1,\ldots,I_k$.
Then the derivative order is $k$ and the sign is $(-1)^k$, so we get
\begin{align*} 
0={}&\sum_{k,\ell\geq0}(-1)^k  \!\!\!\!\!\!\!\!\!\!\!
\sum_{\substack{I_1\sqcup\cdots\sqcup I_k\sqcup J_1\sqcup\cdots\sqcup J_\ell=[n]\\I_i,J_j\neq\emptyset}} \!\!\!\!\!\!\!
\Theta^{(k)}\bigg(\sum_{j=1}^\ell|z_{J_j}|\bigg)
F^\bullet(|z_{J_1}|,\ldots,|z_{J_\ell}|)\bigg(\prod_{i=1}^kV(z_{I_i})\bigg)
\bigg(\prod_{j=1}^\ell V(z_{J_j})\bigg).
\end{align*}
Here $I_1,\ldots,I_k$ and $J_1,\ldots,J_\ell$ are unordered within each group.
Set $J=J_1\sqcup\cdots\sqcup J_\ell$.
For fixed $J$, 
$$
\sum_{\ell\geq0}
\sum_{\substack{J_1\sqcup\cdots\sqcup J_\ell=J\\J_j\neq\emptyset}}
F^\bullet(|z_{J_1}|,\ldots,|z_{J_\ell}|)
\prod_{j=1}^\ell V(z_{J_j})
=(F^\bullet\star V)(z_J).
$$
Therefore, the preceding sum can be rewritten as
\begin{equation}\label{eq:Sigma-CF}
0=\sum_{k\geq0}(-1)^k
\sum_{\substack{I_1\sqcup\cdots\sqcup I_k\sqcup J=[n]\\I_1,\ldots,I_k\neq\emptyset}}
\Theta^{(k)}(|z_J|)(F^\bullet\star V)(z_J)
\bigg(\prod_{i=1}^kV(z_{I_i})\bigg).
\end{equation}

Assume that equation~\eqref{eq:C-star-Fbullet} holds for fewer than $n$ variables.
For $k\geq1$, we have $J\subsetneq[n]$ and thus $(F^\bullet\star V)(z_J)=1$ by induction.
Separating the term with $k=0$ in~\eqref{eq:Sigma-CF}, we have
$$ 
\begin{aligned}
0={}&\Theta(z_1+\cdots+z_n)(F^\bullet\star V)(z_1,\ldots,z_n) +\sum_{k\geq1}(-1)^k \!\!\!\!
\sum_{\substack{I_1\sqcup\cdots\sqcup I_k\sqcup J=[n]\\I_1,\ldots,I_k\neq\emptyset}}\!\!\!\!
\Theta^{(k)}(|z_J|)\bigg(\prod_{i=1}^kV(z_{I_i})\bigg).
\end{aligned}
$$
By equation~\eqref{eqn:C-Theta}, the second sum is $-\Theta(z_1+\cdots+z_n)$.
Hence
$$
\Theta(z_1+\cdots+z_n)\big((F^\bullet\star V)(z_1,\ldots,z_n)-1\big)=0.
$$
Since $\Theta(z_1+\cdots+z_n)\ne0$, we obtain $(F^\bullet\star V)(z_1,\ldots,z_n)=1$.
This proves equation~\eqref{eq:C-star-Fbullet}, and thus $F\star V=\mathscr I$ and equation~\eqref{eqn:relation-F-and-C} hold.
\end{proof}

\begin{proof}[Proof of Theorem~\ref{main_theorem}]
It follows immediately from Propositions~\ref{prop:LRC-1}, \ref{prop:blowup-formal-action} and~\ref{prop:F-star-C}.
\end{proof}

$$
$$

\section{Examples of local del Pezzo surfaces}\label{sec:examples}

We specialize the blowup equation to the following three surfaces:
\[
S=\mathbb P^2,\qquad
\mathbb P^1\times\mathbb P^1,\qquad
\operatorname{Bl}_{p_1,\ldots,p_4}\mathbb P^2,
\]
where the four points in the last surface are in general position.
In each case, let $E\in|-K_S|$ be smooth.

For $\mathbb P^2$, let $H$ be the class of a line.  For
$\mathbb P^1\times\mathbb P^1$, let $B,F$ be the ruling classes,
so that $B^2=F^2=0$ and $B\cdot F=1$.  On the four-point blowup,
write $H$ for the pullback of a line and $e_i$ for the exceptional
classes, with $H^2=1$, $H\cdot e_i=0$, and
$e_i\cdot e_j=-\delta_{ij}$.
{In these curve bases, set $\mathbf r=(E\cdot\beta_1,\ldots,E\cdot\beta_b)$.
This is the vanishing $\mathbf r$-field used in Appendix~\ref{subsec:BPS-blowup-equation}.}
These data are
\begin{equation}\label{eq:examples-geometric-data}
\begin{array}{c|c|c|c|c}
S&E&E\cdot E&\chi(S)& {\mathbf r}\\ \hline
\mathbb P^2&3H&9&3&(3)\\
\mathbb P^1\times\mathbb P^1&2B+2F&8&4&(2,2)\\
\operatorname{Bl}_{4}\mathbb P^2&3H-e_1-e_2-e_3-e_4&5&7&(3,1,1,1,1)
\end{array}
\end{equation}
Its entries for the exceptional classes are positive, since
$E\cdot e_i=1$.

Set $Q_H=\mathbf Q^H$, $Q_B=\mathbf Q^B$,
$Q_F=\mathbf Q^F$, and $Q_i=\mathbf Q^{e_i}$ wherever applicable.
The coordinates adapted to the anticanonical translation are
\begin{equation}\label{eq:examples-adapted-coordinates}
\begin{array}{c|c|c|c}
S&Q&\text{mass coordinates}&D\\ \hline
\mathbb P^2&Q_H&\text{none}&3Q\partial_Q\\[2pt]
\mathbb P^1\times\mathbb P^1&Q_B&{m}=Q_F/Q_B&2Q\partial_Q\\[2pt]
\operatorname{Bl}_{4}\mathbb P^2&Q_1&
\begin{gathered}
{m_0}=Q_H/Q_1^3,\\[-2pt]
{m_i}=Q_i/Q_1\quad(i=2,3,4)
\end{gathered}&Q\partial_Q
\end{array}
\end{equation}
Indeed, in the original coordinates the last two operators are
$2Q_B\partial_{Q_B}+2Q_F\partial_{Q_F}$ and
$3Q_H\partial_{Q_H}+\sum_{i=1}^4Q_i\partial_{Q_i}$,
respectively.  They annihilate the mass coordinates in the table.

Throughout this section we keep the summation convention of
\eqref{eq:blowup-local-relative}.
The hatted potentials include the classical and genus-one terms fixed in
Section~\ref{subsec:normalization}.
In the $F^+$ forms below, the potentials are expressed in the adapted
coordinates of~\eqref{eq:examples-adapted-coordinates}, and all $Q$-derivatives
are taken at fixed masses. We divide out the common factors independent of $N$
arising from the classical and genus-one terms.

\subsection{Local $\mathbb P^2$}\label{subsec:example-P2}
 
With $Q=Q_H$, equation~\eqref{eq:blowup-local-relative} becomes
\begin{equation}\label{eq:examples-P2-full}
\begin{aligned}
0=\sum_{N\in\Z}(-1)^N\exp\Big[
&-3(N+\tfrac12)Q\partial_Q\widehat F^{\mathbb P^2/E}(Q;\hbar) +\widehat F^{K_{\mathbb P^2}}
   \big(e^{-3(N+\tfrac12)\hbar}Q;\mathrm i\hbar\big)\Big].
\end{aligned}
\end{equation}
Equivalently, in terms of the positive-degree potentials,
\begin{equation*}\label{eq:examples-P2-stable} \small
\begin{aligned}
0={}&\sum_{N\in\Z}(-1)^N Q^{\frac32(N+\tfrac12)^2}
\exp\Bigg[-\frac32\hbar(N+\tfrac12)^3+\frac38\hbar(N+\tfrac12)\\
&\quad+\sum_{g\geq0}(-1)^{g-1}\hbar^{2g-2}
 F_g^{K_{\mathbb P^2},+}\big(-e^{-3(N+\tfrac12)\hbar}Q\big)-3(N+\tfrac12)\sum_{g\geq0}\hbar^{2g-1}Q\partial_Q
 F_g^{\mathbb P^2/E,+}(-Q)\Bigg].
\end{aligned}
\end{equation*}

\subsection{Local $\mathbb P^1\times\mathbb P^1$}\label{subsec:example-P1P1}

In the ruling coordinates, the equation takes the form
\begin{equation}\label{eq:examples-P1P1-full}
\begin{aligned}
0=\sum_{N\in\Z}(-1)^N\exp\Big[
&-2(N+\tfrac12)(Q_B\partial_{Q_B}+Q_F\partial_{Q_F})
 \widehat F^{S/E}(Q_B,Q_F;\hbar)\\
&+\widehat F^{K_S}
 \big(e^{-2(N+\tfrac12)\hbar}Q_B,e^{-2(N+\tfrac12)\hbar}Q_F;\mathrm i\hbar\big)\Big].
\end{aligned}
\end{equation}
With $Q=Q_B$ and $m=Q_F/Q_B$, its $F^+$ form is
\begin{equation*}\label{eq:examples-P1P1-stable} {\small
\begin{aligned} 
& 0={}\sum_{N\in\Z}(-1)^N (Q^4m^2)^{\frac12(N+\tfrac12)^2}
\exp\Bigg[-\frac43\hbar(N+\tfrac12)^3+\frac13\hbar(N+\tfrac12)\\
&\qquad+\sum_{g\geq0}(-1)^{g-1}\hbar^{2g-2}
 F_g^{K_S,+}\big(e^{-2(N+\tfrac12)\hbar}Q;m\big) -2(N+\tfrac12)\sum_{g\geq0}\hbar^{2g-1}Q\partial_Q
 F_g^{S/E,+}(Q;m)\Bigg].
\end{aligned} }
\end{equation*}

\subsection{A non-toric example: a degree-five del Pezzo surface}\label{subsec:example-dP5}

For the four-point blowup, we obtain
\begin{equation}\label{eq:examples-dP5-full} 
\begin{aligned}
\!\!\!\! 0=\sum_{N\in\Z}(-1)^N\exp\Bigg[
&-(N+\tfrac12)\left(3Q_H\partial_{Q_H}+\sum_{i=1}^4Q_i\partial_{Q_i}\right)
 \widehat F^{S/E}(Q_H,Q_1,\ldots,Q_4;\hbar)\\
&+\widehat F^{K_S}
 \big(e^{-3(N+\tfrac12)\hbar}Q_H,e^{-(N+\tfrac12)\hbar}Q_1,\ldots,e^{-(N+\tfrac12)\hbar}Q_4;
      \mathrm i\hbar\big)\Bigg].
\end{aligned}
\end{equation}
Using $Q=Q_1$ and the mass coordinates in~\eqref{eq:examples-adapted-coordinates},
we obtain
\begin{equation*}\label{eq:examples-dP5-stable}  \small
\begin{aligned}
0={}&\sum_{N\in\Z}(-1)^N
\left(\frac{Q^5m_0^3}{m_2m_3m_4}\right)^{\frac12(N+\tfrac12)^2}
\exp\Bigg[-\frac56\hbar(N+\tfrac12)^3+\frac5{24}\hbar(N+\tfrac12)\\
&\qquad+\sum_{g\geq0}(-1)^{g-1}\hbar^{2g-2}
 F_g^{K_S,+}\big(-e^{-(N+\tfrac12)\hbar}Q;m_0,m_2,m_3,m_4\big)\\
&\qquad-(N+\tfrac12)\sum_{g\geq0}\hbar^{2g-1}Q\partial_Q
 F_g^{S/E,+}(-Q;m_0,m_2,m_3,m_4)\Bigg].
\end{aligned}
\end{equation*}

The calculation for the four-point blowup therefore uses the same
anticanonical translation and normalization without a toric presentation.

The displayed Gaussian factors are $(\mathbf Q^E)^{\frac12(N+\tfrac12)^2}$.
Combining them with the quadratic Taylor term of the local genus-zero series
gives $\mathbf q^{\frac12(N+\tfrac12)^2}$, where
\begin{equation}\label{eq:examples-common-nome}
\mathbf q=e^{\cT_{0,2}}
=\mathbf Q^E\exp\!\left(-D^2F_0^{K_S,+}\big((-1)^E\mathbf Q\big)\right).
\end{equation}
This is the Gaussian factor in the theta reformulation of
Section~\ref{subsec:geometric-NS-equation}.

$$
$$

\begin{appendices}

\section{Explicit formulas for the universal inverse coefficients}
\label{app:elliptic-stationary}

In~\S\ref{sec:inverse_LRC}, we constructed the universal inverse coefficients from the convolution inverse $\mathscr G$ of the normalized stationary GW family $\mathscr F$ of the elliptic curve $E$ and established their divisor equation and quasimodular weights. In this appendix, we give three representations of the associated generating series $G$.
 Recall that
\[
\begin{gathered}
F(z_1,\ldots,z_n)=\frac{\mathscr F(z_1,\ldots,z_n)}{z_1+\cdots+z_n},\qquad
G(z_1,\ldots,z_n)
=z_1\cdots z_n\,\mathscr G(z_1,\ldots,z_n),
\end{gathered}
\]
Since $\mathscr G=\mathscr F^{\star-1}$, we have $G=F^{\star-1}$.
In this {appendix}, we first express $G$ as a $\mathbf q$-bracket, then obtain two finite tree formulas: one directly from the theta recursion, and the other using the Bloch--Okounkov formula.

\subsection{Partition formulas for  stationary  generating  functions of the elliptic curve}
\label{app:induced-connected-products}
Let $E$ be a nonsingular complex elliptic curve. We denote by $\langle \omega\psi^{a_1} \dots \omega\psi^{a_n} \rangle_{g,n,d}^E$ the genus-$g$ degree-$d$ Gromov--Witten invariants with insertions $\{\omega\psi^{a_i}\}_{i=1}^{n}$, where $\omega$ is the point class and $\psi$ is the psi-class.
See~\cite{OP2006} and references therein for more details.

For $n\geq1$, the connected stationary series is given by
$$
	F(z_1, \dots, z_n) =\delta_{n,1}\frac{1}{z_1}+ \sum_{g \geq 1} \sum_{d \geq 0} \sum_{a_1, \dots, a_n \geq 0} \mathbf q^d\cdot \langle \omega\psi^{a_1} \dots \omega\psi^{a_n} \rangle_{g,n,d}^E \prod_{i=1}^n z_i^{a_i+1},
$$
Its disconnected version, as in Section~\ref{subsec:NS-local-relative-comparison}, is given by
$$
	F^\bullet(z_1, \dots, z_n) := \sum_{k>0} \sum_{\{I_1, \dots, I_k\} \in \mathcal P_n,\,  I_i \neq \emptyset} \prod_{j=1}^k F(z_{I_j}),
$$
where $\mathcal P_n$ is the set of unordered set partitions of $\{1,\cdots,n\}$.
We use the empty-input convention $F^\bullet(z_\emptyset):=1$.
Zagier's relation~\eqref{Zagier_recursive_relation} then determines $F^\bullet$ recursively.
Conversely, for $n\geq 1$, we have
\begin{equation}\label{eq:appq-connected-from-BO}
	F(z_1, \dots, z_n) = \sum_{\pi \in \mathcal{P}_n} (-1)^{|\pi|-1} (|\pi|-1)! \prod_{I \in \pi} F^\bullet(z_I).
\end{equation}

\begin{definition}
{For a family $f=(f_\lambda)_\lambda$ indexed by integer partitions, including the empty partition, its $\mathbf q$-bracket is the normalized sum}
\[
\langle f\rangle_{\mathbf q}
:=\frac{\sum_{\lambda}\mathbf q^{|\lambda|}f_\lambda}
{\sum_{\lambda}\mathbf q^{|\lambda|}}.
\]
Here $|\lambda|$ is the sum of the parts of $\lambda$.
The denominator has constant term $1$, so the quotient is defined formally.
\end{definition}

Dividing the trace formula of Okounkov--Pandharipande~\cite[Equation~(5.2)]{OP2006} by its zero-point series $\sum_\lambda\mathbf q^{|\lambda|}$ removes connected components without markings. With our convention $F^\bullet(z_\emptyset)=1$, this gives the normalized bracket formula (see also~\cite[Equation~(27)]{Zagier2016}):
\begin{equation}\label{eq:appq-moment-bracket}
	F^\bullet(z_1,\cdots,z_n)=\big\langle\mathcal E(z_1)\cdots\mathcal E(z_n)\big\rangle_{\mathbf q}.
\end{equation}
Here $\mathcal E(z)=(\mathcal E_{\lambda}(z))_{\lambda}$ is defined by
\beq\label{eqn:W}
 \mathcal E_\lambda(z):=\mathcal E_\emptyset(z)+\sum_{j=1}^{\ell(\lambda)}\big(e^{(\lambda_j-j+\frac12)z}-e^{(-j+\frac12)z}\big),\qquad
 \mathcal E_\emptyset(z)=\frac{1}{2\sinh(z/2)},
\eeq
and $(\mathcal E(z_1)\cdots\mathcal E(z_n))_{\lambda}=\mathcal E_{\lambda}(z_1)\cdots\mathcal E_{\lambda}(z_n)$.

To represent the connected series by a bracket, we use a product on functions on partitions for which the bracket is multiplicative.
In~\cite{vI2021}, van Ittersum introduced the induced product $\odot$ and the connected product $f_1|\cdots|f_n$ on functions on partitions.
For two functions $f$ and $g$ on partitions, their induced product $\odot$ is defined by
\begin{equation}\label{eq:appq-diamond}
 ( f\odot  g)_\lambda
 :=\sum_{\alpha\cup\beta\cup\gamma=\lambda}
 f_\alpha \cdot  g_\beta \cdot  \mu_\gamma,
\end{equation}
where the union is taken as multisets, and the M\"obius coefficients are
\[
 \mu_\lambda
 :=\begin{cases}
 (-1)^{\ell(\lambda)},&\text{if all parts of $\lambda$ are distinct},\\
 0,&\text{otherwise}.
 \end{cases}
\]
This product is commutative and associative, with the constant function $1$ as its unit.
Moreover, van Ittersum proved~\cite[Proposition~3.2.5]{vI2021} that
\begin{equation}\label{eq:appq-diamond-multiplicative}
 \langle f_1\odot\cdots\odot f_n\rangle_{\mathbf q}
 =\prod_{i=1}^n\langle f_i\rangle_{\mathbf q}.
\end{equation}
Define the connected product by~\cite[Definition~3.2.4]{vI2021}
\begin{equation}\label{eq:appq-connected-observable}
 f_1|\cdots|f_n:=\sum_{\pi\in\mathcal P_n}(-1)^{|\pi|-1}(|\pi|-1)!
 \bigodot_{I\in\pi}\bigg(\prod_{i\in I}f_i\bigg).
\end{equation}
In particular, $f|g=fg-f\odot g$. Multiplicativity under $\odot$ gives
$$
\<f_1|\cdots|f_n\>_{\mathbf q}=\sum_{\pi\in\mathcal P_n}(-1)^{|\pi|-1}(|\pi|-1)!
\prod_{I\in\pi}\bigg\<\prod_{i\in I}f_i\bigg\>_{\mathbf q}
$$
Applying this identity to $\mathcal E(z_i)$ and using~\eqref{eq:appq-moment-bracket}, we obtain
\begin{equation}\label{eq:appq-connected-observable-bracket}
 F(z_1,\ldots,z_n)
 =\big\langle\mathcal E(z_1)|\cdots|\mathcal E(z_n)\big\rangle_{\mathbf q}.
\end{equation}

\subsection{A $\mathbf q$-bracket representation of $G$}
\label{app:observable-convolution}

We lift the inverse relation $F\star G=\mathscr I$ to functions on partitions.
Let $\mathbb A$ be the commutative $\Q$-algebra of $\Q$-valued functions on partitions, with product $\odot$.
\begin{definition}
	Let $\mathcal A$ and $\mathcal B$ be families of symmetric formal power series in $z_1,\ldots,z_n$, for $n\geq1$, with coefficients in $\mathbb A$. We define
	\begin{equation}
		\label{eq:appq-star-q}
		(\mathcal A\circledstar\mathcal B)(z_1,\ldots,z_n)
		:=\sum_{\pi=\{I_1,\ldots,I_k\}\in\mathcal P_n}
		\mathcal A(|z_{I_1}|,\ldots,|z_{I_k}|)
		\odot\bigg(\bigodot_{i=1}^k\mathcal B(z_{I_i})\bigg).
	\end{equation}
\end{definition}
Since the induced product $\odot$ is commutative and associative, as for the ordinary convolution, one can prove that the convolution $\circledstar$ is associative with identity
$\mathscr I(z_1,\ldots,z_n)=\delta_{n,1}$, which is constant with respect to partitions $\lambda$.
Such a formal power-series family $\mathcal A$ is invertible with respect to $\circledstar$ if and only if its unary series $\mathcal A(z)$ is invertible under $\odot$ in the same power-series algebra, with coefficients given by functions on partitions.

By the multiplicativity~\eqref{eq:appq-diamond-multiplicative},
\begin{equation}
	\label{eq:appq-star-q-intertwining}
	\langle\mathcal A\circledstar\mathcal B\rangle_{\mathbf q}
	=\langle\mathcal A\rangle_{\mathbf q}
	\star
	\langle\mathcal B\rangle_{\mathbf q}.
\end{equation}
This identity reduces the construction of $G$ to inversion under $\circledstar$.
\begin{lemma}
	The function $\mathcal E(z)$ defined by~\eqref{eqn:W} is invertible under $\odot$ as a formal Laurent series with coefficients given by functions on partitions.
\end{lemma}
\begin{proof}
	Since $z\mathcal E(z)=1+O(z^2)$ and the constant function $1$ is the unit of $\odot$, the formal geometric series gives $(z\mathcal E(z))^{\odot-1}=1+O(z^2)$. Thus $\mathcal E(z)^{\odot-1}=z(z\mathcal E(z))^{\odot-1}=z+O(z^3)$.
\end{proof}

Now we define 
$$
\mathcal F(z_1,\cdots,z_n)=\mathcal E(z_1)|\cdots |\mathcal E(z_n).
$$
Then $\mathcal F$ is $\circledstar$-invertible since $\mathcal E(z_1)$ is $\odot$-invertible.
Above discussions gives us:
\begin{proposition}\label{thm:appq-inverse-bracket}
Let $\mathcal G$ be defined by $\mathcal G(z)=\mathcal E(z)^{\odot-1}$, and, for $n\geq2$, 
\begin{equation}\label{eq:appq-observable-G-recursion-higher}
	\mathcal G(z_1,\ldots,z_n)=-\mathcal E(z_1+\cdots+z_n)^{\odot-1}\odot  \!\!\!\!\!\!\!\!\!
	\sum_{\substack{\pi=\{I_1,\ldots,I_k\}\in\mathcal P_n\\k\geq2}}  \!\!\!\!\!\!\!
	\mathcal F(|z_{I_1}|,\ldots,|z_{I_k}|)
	\odot\bigg(\bigodot_{i=1}^k\mathcal G(z_{I_i})\bigg).
\end{equation}
Then $\mathcal G=\mathcal F^{\circledstar-1}$ and
$$
G(z_1,\cdots,z_n)=\<\mathcal G(z_1,\cdots,z_n)\>_{\mathbf q}.
$$
\end{proposition}
For example, $\left\langle\mathcal E(z)^{\odot-1}\right\rangle_{\mathbf q}=\Theta(z)$.
For $n=2$, we have
$$
\mathcal G(z_1,z_2)=-\mathcal G(z_1+z_2)\odot\mathcal F(z_1,z_2)\odot\bigg(\bigodot_{i=1}^2\mathcal G(z_i)\bigg),
$$
and 
$$ 
\<\mathcal G(z_1,z_2)\>_{\mathbf q}=-\Theta(z_1+z_2)F(z_1,z_2)\Theta(z_1)\Theta(z_2)
=G(z_1,z_2).
$$

\subsection{A theta-function formula for the inverse coefficients}
\label{app:inverse-elliptic-kernel}
{A second representation follows directly from the theta recursion~\eqref{eqn:recursion-C}. Since $G=V$ by Proposition~\ref{prop:F-star-C}, expanding this recursion gives a finite sum over rooted trees.}

As in the rooted-tree inversion of Guo--Wu~\cite[Corollary~2.2]{GuoWu2025}, the trees record repeated substitution. Their expansion concerns the local-relative correspondence for fixed curve classes, whereas here we expand the generating series $G$ for the universal inverse coefficients, with vertex weights given by derivatives of $\Theta$.

\begin{definition}
For a nonempty set $I\subseteq[n]$, let $\mathfrak T_\Theta(I)$ be the set of isomorphism classes of finite
non-plane rooted trees $T$ together with subsets
$\{J_v\}_{v\in V(T)}$ such that
\begin{enumerate}
\item the subsets $J_v$ are pairwise disjoint and
      $\bigsqcup_{v\in V(T)}J_v=I$;
\item if $J_v=\emptyset$, then the number $d(v)$ of children of $v$
      satisfies $d(v)\geq2$.
\end{enumerate}
Isomorphisms preserve the root and every label in $I$. We use $|z_\emptyset|=0$.
The second condition in particular excludes an empty-labelled leaf and
an empty-labelled unary vertex.  It also implies that every rooted
subtree carries at least one label from $I$.
\end{definition}

\begin{proposition}
\label{thm:appq-theta-tree}
For every $n\geq1$, the convolution inverse $G$ satisfies
\begin{equation}
\label{eq:appq-G-tree-closed}
 G(z_1,\ldots,z_n)
 =\sum_{(T,\{J_v\})\in\mathfrak T_\Theta([n])}
 { \prod_{v\in V(T)}
  (-1)^{d(v)}\Theta^{(d(v))}(|z_{J_v}|)}.
\end{equation}
The sum is finite for fixed $n$.
\end{proposition}

Since $\sum_v d(v)=|E(T)|$, the total sign is also
$(-1)^{|E(T)|}$.

\begin{proof}
Decompose a tree in $\mathfrak T_\Theta([n])$ at its root.  If the root
has no children, it must carry all labels in $[n]$ and contributes
$\Theta(z_1+\cdots+z_n)$.  If it has $k\geq1$ children, let $J$ be the label set of
the root and let $I_1,\ldots,I_k$ be the sets of labels in the child
subtrees.  The root contributes $(-1)^k\Theta^{(k)}(|z_J|)$, while the
child subtrees contribute $\prod_{i=1}^kG(z_{I_i})$.  The admissibility
condition is exactly $J\neq\emptyset$ when $k=1$; for $k\geq2$ the root
may be empty-labelled.  Thus the root decomposition of the right-hand
side of~\eqref{eq:appq-G-tree-closed} is precisely
\eqref{eqn:recursion-C}.  Since the latter determines $G(z_1,\ldots,z_n)$
uniquely by induction on $n$, the formula follows.

For finiteness, observe that every empty-labelled vertex has at least
two children.  Hence the number of such vertices is bounded by the
number of labelled leaves minus one, while there are at most $n$
nonempty vertex labels.
\end{proof}

The first two cases of the tree formula are
\begin{equation}
\label{eq:appq-G1}
 G(z)=\Theta(z),
\end{equation}
and
\begin{equation}
\label{eq:appq-G2}
 G(z_1,z_2)
 =\Theta(z_1+z_2)
 -\Theta'(z_1)\Theta(z_2)
 -\Theta'(z_2)\Theta(z_1).
\end{equation}
By the definition~\eqref{eqn:G_series}, we extract the inverse coefficients from the tree formula~\eqref{eq:appq-G-tree-closed} as follows:
\begin{equation}
\label{eq:appq-G-coefficient-extraction}
 G^E_{g,\ba}
 =\left[z_1^{a_1+2}\cdots z_n^{a_n+2}\right]
 \frac{G(z_1,\ldots,z_n)-\delta_{n,1}z_1}
 {z_1+\cdots+z_n},
 \qquad |\ba|=2g-2.
\end{equation}

\subsection{An expansion using the Bloch--Okounkov formula}
The Bloch--Okounkov formula~\cite[Theorem~6.1]{BO2000}, together with the normalized bracket identity~\eqref{eq:appq-moment-bracket}, gives an explicit expression for $F^\bullet$. We use it below to expand the series $G$.
\begin{theorem}
	\label{thm_ellipticGW}
	For every $n \geq 1$, we have
	$$
	{F^\bullet(z_1,\cdots,z_n)}  =  \sum_{\sigma\in S_n}
	\frac{\det \Big[\frac{\Theta^{(j-i+1)}(z_{\sigma(1)}+\cdots+z_{\sigma(n-j)})}{(j-i+1)!}\Big]_{i,j=1}^n}{\Theta(z_{\sigma(1)})\Theta(z_{\sigma(1)}+z_{\sigma(2)})\cdots \Theta(z_{\sigma(1)}+\cdots+z_{\sigma(n)})}  
	$$
	Here $\Theta(z)=\Theta(z\mid\tau)$ is the same normalized odd theta function as in~\eqref{eq:delPezzo-Theta-expansion}; in particular, $\Theta'(0)=1$.
	The symbol $\Theta^{(k)}$ denotes the $k$-th derivative with respect to $z$, $\frac{\Theta^{(j-i+1)}}{(j-i+1)!}$ is interpreted as $0$ if $j-i+1<0$, and $z_{\sigma(1)}+\cdots+z_{\sigma(n-j)}=0$ for $n=j$.
\end{theorem}

Since $F(z)=1/\Theta(z)$, separating the one-block partition in $F\star G=\mathscr I$ gives, for $n\geq2$,
\begin{equation}
\label{eq:appq-G-connected-recursion}
 G(z_1,\ldots,z_n)
 =-\Theta(z_1+\cdots+z_n)
 \sum_{\substack{\{I_1,\ldots,I_k\}\in\mathcal P_n\\k\geq2}}
 F(|z_{I_1}|,\ldots,|z_{I_k}|)
 \prod_{i=1}^kG(z_{I_i}).
\end{equation}
The arguments indexed by the blocks of a set partition
may be listed in any order because the series are symmetric.

For a nonempty set $I\subseteq[n]$, let $\mathfrak T_{\geq2}(I)$ be the set of isomorphism classes of finite non-plane rooted trees whose
leaves are bijectively labelled by $I$ and whose internal vertices have
at least two children; isomorphisms preserve the root and all leaf labels. If $v$ is a vertex, let $I_v\subseteq I$ be the
set of leaf labels below $v$.  
For each internal vertex $v$, denote its children by $v_1,\ldots,v_{d(v)}$.
\begin{proposition}
\label{thm:appq-connected-tree}
For every $n\geq1$,
\begin{equation}
\label{eq:appq-G-connected-tree}
\begin{aligned}
 G(z_1,\ldots,z_n)
 ={}&\sum_{T\in\mathfrak T_{\geq2}([n])}
 \prod_{v\in V^\circ(T)}
 {\left[-\Theta(|z_{I_v}|)
 F(|z_{I_{v_1}}|,\ldots,|z_{I_{v_{d(v)}}}|)\right]}  
\cdot \prod_{i\in L(T)}\Theta(z_i) .
\end{aligned}
\end{equation}
Here $L(T)$ is identified with the leaf-label set $[n]$, and $V^\circ(T)$ denotes the set of internal vertices.
\end{proposition}

\begin{proof}
For $n=1$, the right-hand side is the single leaf contribution
\eqref{eq:appq-G1}.  For $n\geq2$, decomposing a tree at its root
amounts to choosing a set partition $\{I_1,\ldots,I_k\}\in\mathcal P_n$ with $k\geq2$,
assigning one child subtree to each block, and multiplying by
the root factor $-\Theta(z_1+\cdots+z_n)F(|z_{I_1}|,\ldots,|z_{I_k}|)$.  This is
exactly the recursion~\eqref{eq:appq-G-connected-recursion}.
\end{proof}

For each internal vertex $v$, the M\"obius inversion formula~\eqref{eq:appq-connected-from-BO} expands its $F$-factor as
\begin{equation}
\label{eq:appq-BO-vertex-weight}
\begin{aligned}
F(|z_{I_{v_1}}|,\ldots,|z_{I_{v_{d(v)}}}|)
 ={}&
 \sum_{\pi\in\mathcal P_{d(v)}}
 (-1)^{|\pi|-1}(|\pi|-1)!  \cdot 
 \prod_{J\in\pi}
 F^\bullet\bigl((|z_{I_{v_j}}|)_{j\in J}\bigr).
\end{aligned}
\end{equation}
With this substitution, the finite tree sum~\eqref{eq:appq-G-connected-tree} expresses $G$ in terms of $F^\bullet$. Each factor $F^\bullet$ is given by Theorem~\ref{thm_ellipticGW}, so this is a non-recursive formula in Bloch--Okounkov determinants.

\begin{remark}
Equations~\eqref{eq:appq-G-tree-closed} and
\eqref{eq:appq-G-connected-tree} express the substitution inverse
as rooted-tree sums rather than as a single $n\times n$ determinant
analogous to the original Bloch--Okounkov determinant.  Further
compression would require additional theta identities, such as
systematic cancellations arising from Fay identities for the
Kronecker function.

\end{remark}

$$
$$

\section{Refined BPS theory and the NS limit}

\label{app:refined-BPS}

We review the refined BPS free energy and its conjectural identifications with the local and relative GW series in the topological and NS limits, respectively. Taking the NS limit of the conjectural refined vanishing blowup equation then leads to the equation~\eqref{eq:blowup-local-relative}.

\subsection{Refined BPS free energy and its topological and Nekrasov--Shatashvili limits} 
In this subsection, we recall the refined BPS invariants.
The topological and Nekrasov--Shatashvili (NS) limits of the refined BPS free energy are conjecturally identified with the local and relative GW generating series, respectively.

Refined BPS invariants record the multiplicities of BPS particles with prescribed curve classes and spins.
In the M-theory interpretation of Gopakumar and Vafa~\cite{GV1998a,GV1998b}, these particles arise from M2-branes wrapped on holomorphic curves in a local Calabi--Yau threefold $X$.
Their spins are labelled by representations of the five-dimensional massive little group $SU(2)_L\times SU(2)_R$.
In this paper, we take $X=\operatorname{Tot}(K_S)$.
For a nonzero effective curve class $\beta$ and $j_L,j_R\in\frac12\mathbb Z_{\geq0}$, we denote by
$$N_\beta^{j_L,j_R}$$
the corresponding refined BPS invariants.
Related mathematical approaches include refined stable pair invariants~\cite{CKK2012} and Maulik--Toda's vanishing cycle construction of Gopakumar--Vafa invariants~\cite{MT2016}.

The refined BPS free energy consists of the perturbative part (corresponding to the classical and unstable parts in the local and relative GW theories) and the instanton part.
We first consider the instanton part.
Let $t$ denote the K\"ahler parameters, with $\beta\cdot t$ their pairing with the curve class $\beta$. We adopt the conventions of~\cite{HSW2018} and write the BPS instanton free energy as
\begin{equation}\label{eq:delPezzo-refined-instanton-part}
	F^{\rm ins}(t;\epsilon_1,\epsilon_2)
	=\sum_{j_L,j_R\geq0}\sum_{\beta}\sum_{k=1}^{\infty}
	\frac{(-1)^{2j_L+2j_R}N_\beta^{j_L,j_R}\, \chi_{j_L}(\cq_L^k)\chi_{j_R}(\cq_R^k)}
	{k(\cq_1^{k/2}-\cq_1^{-k/2})(\cq_2^{k/2}-\cq_2^{-k/2})}
	e^{-k\beta\cdot t},
\end{equation}
where the summation over $\beta$ runs over nonzero effective curve classes and
the summation over $j_L,j_R$ runs over nonnegative half-integers. 
Here $\cq_{1,2}=e^{\epsilon_{1,2}}$, $\cq_{L,R}=e^{(\epsilon_1\mp\epsilon_2)/2}$, and
\[
\chi_j(\cq)=\frac{\cq^{2j+1}-\cq^{-2j-1}}{\cq-\cq^{-1}}.
\]
Furthermore, the instanton free energy has the following structure
$$
F^{\rm ins}(t;\epsilon_1,\epsilon_2)=\sum_{h,g\geq 0} (\epsilon_1+\epsilon_2)^{2h} (\epsilon_1\epsilon_2)^{g-1}F_{\mathrm{ref}}^{(h,g)}(t).
$$

The topological limit of the BPS instanton free energy is defined by taking $\epsilon_1=-\epsilon_2$:
$$
	F^{\mathrm{top},+}(t;\hbar):=F^{\rm ins}(t;-\hbar,\hbar).
$$
This indeed gives the Gopakumar--Vafa (GV) free energy~\cite{GV1998a,GV1998b}.
For local del Pezzo surfaces, the Gromov--Witten/Gopakumar--Vafa correspondence~\cite{GV1998a,GV1998b} predicts the following identification:
\begin{equation}\label{eq:BPS-local-limit-recall}
	F^{\mathrm{top},+}(t;\hbar)
	=\sum_{g\geq0}(-\hbar^2)^{g-1} \sum_{\beta,\beta\cdot E>0}N_{g,\beta}^{K_S}\mathbf Q^\beta
	=\sum_{g\geq 0} (-\hbar^2)^{g-1} F_g^{K_S,+}({\bf Q}),
\end{equation}
where $\mathbf Q^\beta=e^{-\beta\cdot t}$.
Using their vanishing-cycle definition of GV invariants~\cite[Conjecture~3.18]{MT2016},
Maulik--Toda formulated a sheaf-theoretic version of this comparison for projective Calabi--Yau threefolds.
For local del Pezzo surfaces, this sheaf-theoretic correspondence is known in genus zero~\cite[Section~8.6]{Bousseau2020QTV}, while the all-genus correspondence remains conjectural in general.

The NS limit~\cite{NekrasovShatashvili2009} of the BPS instanton free energy is defined by taking the residue at $\epsilon_1=0$; equivalently,
$$
	F^{\mathrm{NS},+}(t;\hbar):=\lim_{\epsilon_1\to0}\epsilon_1\,  F^{\rm ins}(t;\epsilon_1,\hbar).
$$
It is conjecturally related to the relative GW invariants as follows:
\begin{equation}\label{eq:BPS-relative-limit-recall}
	F^{\mathrm{NS},+}(t;\hbar)
	=\sum_{g\geq0} \hbar^{2g-1} \sum_{\beta,\beta\cdot E>0}\frac{(-1)^{\beta\cdot E+g-1}}{\beta\cdot E} N_{g,\beta}^{S/E}\mathbf Q^\beta
	=\sum_{g\geq 0} \hbar^{2g-1} F_g^{S/E,+}({\bf Q}).
\end{equation}
Bousseau proposed a sheaf-theoretic formulation of this correspondence for local del Pezzo surfaces~\cite[Conjecture~8.16]{Bousseau2020QTV}. See also~\cite[Conjecture~1.2]{GuoWu2025}.
For $S=\mathbb P^2$ and $E$ a smooth cubic, Bousseau proved the refined sheaves/GW correspondence~\cite{Bousseau2023}. 
Combined with the cohomological $\chi$-independence theorem of Maulik--Shen~\cite{MaulikShen2023}, this establishes the sheaf-theoretic NS/relative comparison in every positive degree and all genera.
See also~\cite{BFGW2021}.
For general local del Pezzo surfaces, this sheaf-theoretic NS/relative correspondence remains conjectural.

\begin{remark}
	In~\cite{BS2024}, Brini--Schuler constructed the refined GW invariants using equivariant stable maps.
	For smooth del Pezzo pairs, they proved that the refined generating series recovers the local GW series and the normalized relative GW series by taking the topological limit and the NS limit, respectively.
	Moreover, for local $\mathbb P^2$, they proved the refined Gromov--Witten/Gopakumar--Vafa correspondence in the NS limit~\cite[Corollary~7.12]{BS2024}.
	In general, the correspondence between Brini--Schuler's construction of the refined GW invariants and the cohomological refinement of the Maulik--Toda invariants is still conjectural~\cite[Conjecture~7.9]{BS2024}.
\end{remark}

\subsection{Blowup equation for local and relative GW series}\label{subsec:BPS-blowup-equation}

We first recall the perturbative part of the refined BPS free energy.
We fix an integral basis \(\{\beta_1,\ldots,\beta_b\}\) of \(H_2(S,\mathbb Z)\) and use the corresponding K\"ahler coordinates \(t=(t_1,\ldots,t_b)\).
The perturbative free energy is given by
$$
	F^{\rm per}(t;\epsilon_1,\epsilon_2)
	=\frac{1}{\epsilon_1\epsilon_2}\bigg(\frac16 \sum_{i,j,k=1}^{b}a_{ijk}t_it_jt_k
	+4\pi^2 \sum_{i=1}^{b}b_i^{\mathrm{NS}}t_i\bigg)
	+\sum_{i=1}^{b}b_it_i-\frac{(\epsilon_1+\epsilon_2)^2}{\epsilon_1\epsilon_2} \sum_{i=1}^{b}b_i^{\mathrm{NS}}t_i .
$$
The coefficients $a_{ijk}$, $b_i$, and $b^{\rm NS}_i$ specify the perturbative part, which complements the positive-degree BPS expansion.
See~\cite{HSW2018} and the references therein for more details.
In fact, the term $4\pi^2 \sum_{i=1}^{b}b_i^{\mathrm{NS}}t_i$ will not affect the vanishing blowup equation stated below.
After omitting this term, we write the classical and unstable contributions as follows:
$$
F^{\rm top}_{\rm classical}(t)=\frac16\sum_{i,j,k=1}^{b}a_{ijk}t_it_jt_k,\quad
F^{\rm top}_{\rm unstable}(t)=\sum_{i=1}^{b}b_it_i,\quad
F^{\rm NS}_{\rm unstable}(t)=-\sum_{i=1}^{b}b_i^{\rm NS}t_i.
$$
For topological and NS limits of the refined BPS free energy, we define 
\begin{align*}
F^{\rm top}(t;\hbar)=&-\frac{1}{\hbar^2}F^{\rm top}_{\rm classical}(t)+F^{\rm top}_{\rm unstable}(t)+F^{\mathrm{top},+}(t;\hbar),\\
F^{\rm NS}(t;\hbar)=&\frac{1}{\hbar}F^{\rm top}_{\rm classical}(t)+\hbar F^{\rm NS}_{\rm unstable}(t) +F^{\mathrm{NS},+}(t;\hbar).
\end{align*}
We expect the following equations
$$
F^{\rm top}(t;\hbar)=\sum_{g\geq 0} (-\hbar^2)^{g-1} F^{K_{S}}_g({\bf Q}),\qquad
F^{\rm NS}(t;\hbar)=\sum_{g\geq 0} \hbar^{2g-1} F^{S/E}_g({\bf Q}).
$$

Following~\cite{HSW2018,GHKSW2020}, for weak del Pezzo surfaces, the conjectural vanishing blowup equation is given as follows:
\begin{equation}\label{eq:BPS-refined-vanishing}
	\begin{aligned}
		0=\sum_{N\in\Z}(-1)^N\,
		\widehat Z^{\rm ref}\big(t+\epsilon_1 (N+\tfrac12){\bf r};\epsilon_1,\epsilon_2-\epsilon_1\big)\,
		\widehat Z^{\rm ref}\big(t+\epsilon_2 (N+\tfrac12){\bf r};\epsilon_1-\epsilon_2,\epsilon_2\big),
	\end{aligned}
\end{equation}
where ${\bf r}=(E\cdot \beta_1,\cdots,E\cdot \beta_b)$
is the chosen vanishing $\mathbf r$-field: it specifies the half-integer shifts
$(N+\tfrac12)\mathbf r$ in the sum.
The twisted partition function $\widehat Z^{\rm ref}(t;\epsilon_1,\epsilon_2)$ is defined by
\begin{equation}\label{eq:BPS-twisted-free-energy}
\widehat Z^{\rm ref}(t;\epsilon_1,\epsilon_2)
=\exp\bigl(F^{\rm per}(t;\epsilon_1,\epsilon_2)
+F^{\rm ins}(t+\pi\mathrm i {\bf r};\epsilon_1,\epsilon_2)\bigr).
\end{equation}
\begin{remark}
	Let $\mathbf B_{\mathrm{BPS}}=(B_1,\ldots,B_b)\in\Z^b$ be a B-field vector satisfying
	$$
		2j_L+2j_R+1\equiv \mathbf B_{\mathrm{BPS}}\cdot\beta\mod2,
	$$
	whenever $N_\beta^{j_L,j_R}\neq0$.
	Here, $\mathbf B_{\mathrm{BPS}}\cdot\beta=\sum_iB_id_i$ for $\beta=\sum_i d_i\beta_i$.
	The existence of such a vector follows from the checkerboard pattern of the non-vanishing refined BPS invariants~\cite{CKK2012,HMMO2013,HSW2018}.
	The vanishing $\mathbf r$-field must satisfy
	$$
	{\bf r}\equiv \mathbf B_{\mathrm{BPS}} \mod 2\mathbb Z^{b}.
	$$
	For the chosen local del Pezzo vanishing $\mathbf r$-field, we may take
	$\mathbf B_{\mathrm{BPS}}=\mathbf r$~\cite{GHKSW2020}.
	In the curve bases used in Section~\ref{sec:examples}, this vector is
	\begin{equation}\label{eq:examples-bps-vectors}
	\begin{array}{c|c}
	S&\mathbf r\\ \hline
	\mathbb P^2&(3)\\
	\mathbb P^1\times\mathbb P^1&(2,2)\\
	\operatorname{Bl}_{p_1,\ldots,p_4}\mathbb P^2&(3,1,1,1,1).
	\end{array}
	\end{equation}
	The entries for the exceptional classes are positive because $E\cdot e_i=1$.
\end{remark}

We take the NS limit of the vanishing blowup equation by setting $\epsilon_2=\hbar$ and letting $\epsilon_1\to0$.
For the instanton part, the definition of $F^{\mathrm{NS},+}$ gives
$$\textstyle
F^{\rm ins}(t;\epsilon_1,\hbar)=\frac{1}{\epsilon_1} F^{\mathrm{NS},+}(t;\hbar)+O(1).
$$
After dividing the summand in~\eqref{eq:BPS-refined-vanishing} by $\widehat Z^{\rm ref}(t;\epsilon_1,\hbar)$, the two instanton contributions have the following $\epsilon_1\to 0$ limits
\begin{equation}
	\lim_{\epsilon_1\to 0} \big(F^{\rm ins}(\tilde t+\epsilon_1 (N+\tfrac12){\bf r};\epsilon_1,\hbar-\epsilon_1)-F^{\rm ins}(\tilde t;\epsilon_1,\hbar)\big)
	= -(N+\tfrac12)DF^{\mathrm{NS},+}(\tilde t;\hbar)-\partial_\hbar F^{\mathrm{NS},+}(\tilde t;\hbar),
\end{equation}
\begin{equation}
\lim_{\epsilon_1\to 0}F^{\rm ins}(\tilde t+\hbar (N+\tfrac12){\bf r};\epsilon_1-\hbar,\hbar)
= F^{\mathrm{top},+}(\tilde t+\hbar (N+\tfrac12){\bf r};\hbar).
\end{equation}
Here $\tilde t=t+\pi\mathrm i {\bf r}$, and we have used $D=-{\bf r}\cdot\partial_{t}=-\sum_i (E\cdot\beta_i)\partial_{t_i}$ which can be easily seen from $D{\bf Q}^{\beta}=(E\cdot \beta){\bf Q}^{\beta}$.
We can similarly deal with the perturbative part.
After dividing out the factors independent of $N$, we get the following NS limit of the vanishing blowup equation
\begin{equation}\label{eq:BPS-NS-vanishing}
\sum_{N\in\Z}(-1)^N\exp\Big(-(N+\tfrac12)D \widehat F^{\rm NS}(t;\hbar)
+\widehat F^{\rm top}(t+\hbar (N+\tfrac12){\bf r};\hbar)\Big)=0,
\end{equation}
where
\begin{align*}
	\widehat F^{\rm top}(t;\hbar)=&-\frac{1}{\hbar^2}F^{\rm top}_{\rm classical}(t)+F^{\rm top}_{\rm unstable}(t)+F^{\mathrm{top},+}(\tilde t;\hbar),\\
	\widehat F^{\rm NS}(t;\hbar)=&\frac{1}{\hbar}F^{\rm top}_{\rm classical}(t)+\hbar F^{\rm NS}_{\rm unstable}(t)+F^{\mathrm{NS},+}(\tilde t;\hbar).
\end{align*}
In~\cite[Sections~2.5 and~3.2]{HSW2018}, Huang, Sun, and Wang identified the NS limit of vanishing blowup equations with compatibility formulae for quantum mirror curves.
See also~\cite{GG2019,SWH2017}.
To compare this equation with the equation \eqref{eq:blowup-local-relative}, use the twisted potentials~\eqref{eqn:hat-potentials}.
Under the conjectural identifications above, with compatible normalizations of the classical and unstable terms, we have
$$
\widehat F^{\rm top}(t;\hbar)=\widehat F^{K_S}({\bf Q};\mathrm i\hbar),\qquad
\widehat F^{\rm NS}(t;\hbar)=\widehat F^{S/E}({\bf Q};\hbar).
$$
The shift $t\mapsto t+\hbar (N+\tfrac12){\bf r}$ corresponds to ${\bf Q}\mapsto \cq^{(N+\frac12)E}{\bf Q}$, where $\cq=e^{-\hbar}$.
Substituting these identities into~\eqref{eq:BPS-NS-vanishing} gives the equation~\eqref{eq:blowup-local-relative} stated in the introduction.

\end{appendices}

$$
$$

\end{document}